\documentclass[times]{oupau2}

\usepackage[all,2cell]{xy}       
\usepackage[makeroom]{cancel}    

\newcommand{\ggl}{\mathfrak{gl}}                                   
\renewcommand{\gg}{\mathfrak{g}}            				       
\newcommand{\G}{\mathcal{G}}									   
\newcommand{\hh}{\mathfrak{h}}            						   

\newcommand{\Lie}{\mathcal{L}}                                     
\newcommand{\ad}{\textnormal{ad}}           					   
\newcommand{\coker}{\textnormal{coker }}                           
\providecommand{\abs}[1]{\lvert#1\rvert}						   
\newcommand{\rest}[1]{\Big{\vert}_{#1}}							   

\newcommand{\vJoin}{\mathbin{\rotatebox[origin=c]{90}{$\Join$}}}   

\newtheorem{remark}{Remark}										   

\begin{document}

\title{A cohomological proof for the integrability of strict Lie 2-algebras}
\shorttitle{Integrability of Lie 2-algebras}

\volumeyear{2021}
\paperID{rnn999}

\author{Camilo Angulo\affil{1}}
\abbrevauthor{C. Angulo}
\headabbrevauthor{Angulo, C}

\address{%
\affilnum{1} Departamento de Matem\'atica Aplicada, Instituto de Matem\'atica da Universidade Federal Fluminense, Rua Professor Marcos Waldemar de Freitas Reis, s/n, Campus do Gragoat\'a \\
Niter\'oi, Rio de Janeiro 24210-201, Brazil.}

\correspdetails{ca.angulo951@gmail.com}

\received{29 January 2021}


\begin{abstract}
We prove a series of van Est type theorems relating the cohomologies of 
strict Lie 2-groups and strict Lie 2-algebras, and use them to prove the 
integrability of Lie 2-algebras anew.
\end{abstract}

\maketitle

\section{Introduction}\label{sec-Intro}
Associated with any Lie algebra $\gg$, there is a canonical extension 
\begin{eqnarray}\label{adExt}
\xymatrix{
0 \ar[r] & \mathfrak{z}(\gg) \ar@{^{(}->}[r] & \gg \ar[r]^{\ad\quad} & \ad(\gg) \ar[r] & 0
}
\end{eqnarray}
of the linear Lie subalgebra $\ad(\gg)\leq\ggl(\gg)$ by the center 
$\mathfrak{z}(\gg)$. Such extensions are classified by 
$H_{CE}^2(\ad(\gg),\mathfrak{z}(\gg))$, the Chevalley-Eilenberg cohomology 
of $\ad(\gg)$ with values on $\mathfrak{z}(\gg)$. Recall the following 
theorem due to van Est:

\begin{theorem}\cite{vanEst:1953}\label{vanEst}
Let $G$ be a Lie group with Lie algebra $\gg$ and a representation on $V$. 
If $G$ is $k$-connected, then the map that differentiates group cochains 
into Chevalley-Eilenberg cochains induces isomorphisms 
\begin{eqnarray}\label{vanEstIso}
\xymatrix{
\Phi :H_{Gp}^n(G,V) \ar[r] & H_{CE}^n(\gg,V)
}
\end{eqnarray}
for all $n\leq k$ and is injective when $n=k+1$.
\end{theorem}

Observe that, being linear, $\ad(\gg)$ is integrable and recall that one 
can always pick a $2$-connected integration, say $G$. Thus, 
Theorem~\ref{vanEst} implies there exists a unique cohomology class 
$[\int\omega_\gg]\in H_{Gp}^2(G,\mathfrak{z}(\gg))$, whose image under 
(\ref{vanEstIso}) is $[\omega_\gg]\in H_{CE}^2(\ad(\gg),\mathfrak{z}(\gg))$, 
the cohomology class that corresponds to the extension (\ref{adExt}). 
Since $H_{Gp}^2(G,V)$ classifies extensions of $G$ by $V$, there is a 
unique extension 
\begin{eqnarray}\label{intExt}
\xymatrix{
1 \ar[r] & \mathfrak{z}(\gg) \ar[r] & \G \ar[r] & G \ar[r] & 1
}
\end{eqnarray}
corresponding to $[\int\omega_\gg]$, and $\G$ is a Lie group integrating 
$\gg$.

The purpose of this article is to adapt this strategy -that we henceforth 
refer to as the \textit{van Est strategy} \cite{vanEst:1955,Crainic:2003}- 
to prove the integrability of strict Lie $2$-algebras. 

A strict Lie $2$-group is a groupoid 
\begin{eqnarray}\label{ALie2Gp}
\xymatrix{
\G\times_H\G \ar[r]^{\quad m} & \G \ar@<0.5ex>[r]^{s} \ar@<-0.5ex>[r]_{t} \ar@(l,d)[]_{\iota} & H \ar[r]^{u} & \G ,
}
\end{eqnarray}
in which the spaces of objects, arrows and composable arrows are Lie 
groups, and all of whose structural morphisms are Lie group homomorphisms. 
Differentiating the whole structure, one gets a (strict) Lie $2$-algebra
\begin{eqnarray}\label{ALie2Alg}
\xymatrix{
\gg_1\times_\hh \gg_1 \ar[r]^{\quad \hat{m}} & \gg_1 \ar@<0.5ex>[r]^{\hat{s}} \ar@<-0.5ex>[r]_{\hat{t}} \ar@(l,d)[]_{\hat{\iota}} & \hh \ar[r]^{\hat{u}} & \gg_1 .
}
\end{eqnarray}

In \cite{Sheng_Zhu1:2012}, it is proven using the path method that all 
finite-dimensional Lie $2$-algebras are the infinitesimal counterpart of a 
Lie $2$-group. In the sequel, we present a cohomological proof of this fact. 
Such approach still works in infinite dimensions and was historically used 
to construct the first example of a non-integrable Lie algebra 
\cite{vanEst_Korthagen:1964}; thus, it bears the potential to improve our 
current understanding of the Lie theory of other categorified objects (see, 
\textit{e.g.}, \cite{Bursztyn_Cabrera_delHoyo:2016,Stefanini:2008,Neeb:2002}).

Let us list the necessary ingredients for the van Est strategy to run:
\begin{itemize}
\item[1)] The canonically associated adjoint extension (\ref{adExt}).
\item[2)] Global and infinitesimal cohomology theories that classify extensions. 
\item[3)] A van Est map and theorem.
\item[4)] That linear Lie algebras be integrable to $2$-connected Lie groups.
\end{itemize}

Lie $2$-algebras have a canonically associated adjoint representation (see 
Example~\ref{2ad} below). In \cite{Angulo1:2020,Angulo2:2020}, complexes 
whose second cohomology classify respectively extensions of Lie $2$-algebras 
and extensions of Lie $2$-groups are introduced. Each of these is the total 
complex of a triple complex of sorts. In order to describe them, let us fix 
notation. 

First, recall that the categories of Lie $2$-algebras and Lie $2$-groups 
are respectively equivalent to the categories of crossed modules of Lie 
algebras and of Lie groups 
\cite{Baez_Crans:2004,Baez_Lauda:2004,Loday:1982}. 

\begin{definition}\label{AlgCrossMod}
A \textit{crossed module of Lie algebras} is a Lie algebra morphism 
$\xymatrix{\gg \ar[r]^\mu & \hh}$ together with a Lie algebra action by 
derivations $\xymatrix{\Lie:\hh \ar[r] & \ggl(\gg)}$ satisfying
\begin{align*}
\mu(\Lie_y x) & =[y,\mu(x)], & \Lie_{\mu(x_0)}x_1 & =[x_0,x_1]
\end{align*} 
for all $y\in\hh$ and $x,x_0,x_1\in\gg$. Following the convention in the 
literature, we refer to these equations respectively as equivariance and 
infinitesimal Peiffer.
\end{definition} 

\begin{definition}\label{GpCrossMod}
A \textit{crossed module of Lie groups} is a Lie group homomorphism 
$\xymatrix{G \ar[r]^i & H}$ together with a right action of $H$ on $G$ by 
Lie group automorphisms satisfying
\begin{align*}
i(g^h) & =h^{-1}i(g)h, & g_1^{i(g_2)} & =g_2^{-1}g_1g_2,
\end{align*} 
for all $h\in H$ and $g,g_1,g_2\in G$, where we write $g^h$ for $h$ acting 
on $g$. Following the convention in the literature, we refer to these 
equations respectively as equivariance and Peiffer. 
\end{definition}

Representations of both Lie $2$-algebras and Lie $2$-groups take values on 
so-called $2$-vector spaces. These are (flat) abelian objects in either 
category, which, in crossed module presentation, correspond simply to 
$2$-term complexes of vector spaces $\xymatrix{W \ar[r]^\phi & V}$. The 
category of linear invertible self-functors of a $2$-vector space and 
homomorphic natural transformations $GL(\phi)$ has got the structure of a 
Lie $2$-group whose Lie $2$-algebra $\ggl(\phi)$ is the category of linear 
functors and linear natural transformations (see Subsection~\ref{sss-LinAndRep} 
for details). Respectively, representations are by definition maps of Lie 
$2$-groups to $GL(\phi)$ and maps of Lie $2$-algebras to $\ggl(\phi)$. 

Lastly, we assume the following unconventional notation for the spaces of 
$p$-composable arrows:
\begin{align}\label{p-comp}
\G_p  & :=\lbrace(\gamma_1,...,\gamma_p)\in\G^p:s(\gamma_k)=t(\gamma_{k+1}),\quad\forall k\rbrace , & \textnormal{ and }\quad
\gg_p & :=\lbrace(\xi_1,...,\xi_p)\in\gg^p:\hat{s}(\xi_k)=\hat{t}(\xi_{k+1}),\quad\forall k\rbrace .
\end{align}
We are ready to define the three-dimensional lattices of vector spaces 
underlying the definition of the complexes of Lie $2$-algebra and Lie 
$2$-group cochains taking values on the $2$-vector space 
$\xymatrix{W \ar[r]^\phi & V}$. For a Lie $2$-algebra (\ref{ALie2Alg})  
with associated crossed module $\xymatrix{\gg \ar[r] & \hh}$, set  
\begin{eqnarray}\label{Alg3dimLat}
C^{p,q}_r(\gg_1,\phi):=\bigwedge^q\gg_p^*\otimes\bigwedge^r\gg^*\otimes W
& \textnormal{ for } r>0, \textnormal{ and } & 
C^{p,q}_0(\gg_1,\phi):=\bigwedge^q\gg_p^*\otimes V,
\end{eqnarray}
where $\gg_0=\hh$.
For a Lie $2$-group (\ref{ALie2Gp}) with associated crossed module 
$\xymatrix{G \ar[r] & H}$, set
\begin{eqnarray}\label{Gp3dimLat}
C^{p,q}_r(\G,\phi):=C(\G^q_p\times G^r,W)
& \textnormal{ for } r>0, \textnormal{ and } &
C^{p,q}_0(\G,\phi):=C(\G^q_p,V),
\end{eqnarray}
where $\G_0 =H$, $\G_1=\G$ and $C(X,A)$ is the vector space of $A$-valued 
smooth functions.

These lattices come together with a three-dimensional \textit{grid} of maps 
that is a complex in each direction; in the Lie $2$-algebra case, the grid 
is built out of Chevalley-Eilenberg complexes, while in the Lie $2$-group 
case, the grid is built out of groupoid cochain complexes (see 
Subsection~\ref{sss-Cxs} for details). We refrain from calling either grid 
a triple complex because not all differentials commute with one another. In 
each case, two of the building differentials commute only up to homotopy 
(or up to isomorphism when $r=0$). In \cite{Angulo1:2020,Angulo2:2020}, it 
is explained how adding the homotopies to the total differential $\nabla$ 
makes up for this defect, ultimately turning 
\begin{eqnarray}\label{Cxs}
C^{n}_\nabla(\gg_1,\phi)=\bigoplus_{p+q+r=n}C^{p,q}_r(\gg_1,\phi) & \textnormal{ and } & 
C^{n}_\nabla(\G,\phi)=\bigoplus_{p+q+r=n}C^{p,q}_r(\G,\phi)
\end{eqnarray}
into actual complexes.

The fundamental property of the complexes (\ref{Cxs}), as it was mentioned, 
is that their second cohomology classify extensions. In fact, the 
equivalence happens at the level of cocycles, \textit{i.e.}, a $2$-cocycle 
in either $C^{2}_{tot}(\gg_1,\phi)$ or $C^{2}_{tot}(\G,\phi)$ univoquely 
defines an extension by the $2$-vector space $\xymatrix{W \ar[r]^\phi & V}$, 
and two such extensions are isomorphic if and only if the cocycles are 
cohomologous. In so, the map that \textit{linearizes} a Lie $2$-group 
extension induces a map 
\begin{eqnarray}\label{2vanEst2}
\xymatrix{
\Phi:C_\nabla^2(\G,\phi) \ar[r] & C_\nabla^2(\gg_1,\phi)
}
\end{eqnarray} 
whenever $\gg_1$ is the Lie $2$-algebra of $\G$. Due to its nature, we are 
bound to call $\Phi$ the van Est map. This map can be proved to be 
assembled from groupoid van Est maps (see Section~\ref{sec-Theos}). Recall 
that the van Est theorem (Theorem~\ref{vanEst}) admits an extension to 
Lie groupoids.

\begin{theorem}\cite{Crainic:2003}\label{Crainic-vanEst}
Let $\xymatrix{\G \ar@<0.5ex>[r] \ar@<-0.5ex>[r] & M}$ be a Lie groupoid 
with Lie algebroid $A$ and a representation on the vector bundle $E$. If 
the source fibres of $\G$ are $k$-connected, the map that differentiates 
groupoid cochains into algebroid cochains induces isomorphisms
\begin{eqnarray}\label{Crainic-vanEstIso}
\xymatrix{
\Phi:H_{Gpd}^n(G,E) \ar[r] & H_{CE}^n(A,E)
}
\end{eqnarray}
for all $n\leq k$ and it is injective for $n=k+1$.
\end{theorem}

We refer to Theorem~\ref{Crainic-vanEst} as the Crainic-van Est theorem in 
the sequel. The Crainic-van Est theorem can be rephrased as a vanishing 
result for the cohomology of the mapping cone of $\Phi$ (see 
Proposition~\ref{ConeCoh}).

\begin{theorem}\label{Crainic-vanEstRephrased}
If $\G$ is source $k$-connected, then
\begin{eqnarray*}
H^n(\Phi)=(0),\quad\textnormal{for all } n\leq k,
\end{eqnarray*}
where $H^\bullet(\Phi)$ is the cohomology of the mapping cone of $\Phi$.
\end{theorem}
 
One can combine the fact that $\Phi$ is assembled from groupoid van Est 
maps and Theorem~\ref{Crainic-vanEstRephrased} to prove a van Est type 
theorem. To see how, let us momentarily take $W=(0)$. In this case, the 
three-dimensional grids (\ref{Alg3dimLat}) and (\ref{Gp3dimLat}) collapse 
to honest double complexes 

\begin{eqnarray*}
\xymatrix{
\vdots                              & \vdots                                 & \vdots                             &       \\ 
C(H^2,V) \ar[r]^{\partial}\ar[u]  & C(\G^2,V) \ar[r]^{\partial}\ar[u] & C(\G_2^2)\ar[r]\ar[u]               & \dots \\
C(H,V) \ar[r]^{\partial}\ar[u]^\delta & C(\G,V) \ar[r]^{\partial}\ar[u]^\delta    & C(\G_2) \ar[r]\ar[u]^\delta & \dots \\
V \ar[r]^{\partial}\ar[u]^\delta & V \ar[r]^{\partial}\ar[u]^\delta    & V \ar[r]\ar[u]^\delta     & \dots 
} & \xymatrix{ & \\
 			   & \\
			   & \\
   \ar[r]^\Phi & } & \xymatrix{
\vdots                              & \vdots                                 & \vdots                             &       \\ 
\bigwedge^2\hh^*\otimes V \ar[r]^{\partial}\ar[u]  & \bigwedge^2\gg_1^*\otimes V \ar[r]^{\partial}\ar[u] & \bigwedge^2\gg_2^*\otimes V\ar[r]\ar[u]               & \dots \\
\hh^*\otimes V \ar[r]^{\partial}\ar[u]^\delta & \gg_1^*\otimes V \ar[r]^{\partial}\ar[u]^\delta    & \gg_2^*\otimes V \ar[r]\ar[u]^\delta & \dots \\
V \ar[r]^{\partial}\ar[u]^\delta & V \ar[r]^{\partial}\ar[u]^\delta    & V \ar[r]\ar[u]^\delta     & \dots 
}
\end{eqnarray*} 

where $\Phi$ is defined columnwise by 
$\xymatrix{\Phi_p:C_{Gp}^\bullet(\G_p,V) \ar[r] & C_{CE}^\bullet(\gg_p,V)}$, 
the usual van Est map for $\G_p$. Now, $\Phi$ is a map of double complexes 
if and only if 
\begin{align*}
\xymatrix{
C(\Phi_0) \ar[r] & C(\Phi_1) \ar[r] & C(\Phi_2) \ar[r] & \dots
}
\end{align*}
is a double complex, and the first page of the spectral sequence of its 
filtration by columns is
\begin{align}\label{1stPage}
\xymatrix{
& \vdots          & \vdots           & \vdots          &       \\ 
& H^2(\Phi_0) \ar[r] & H^2(\Phi_1) \ar[r] & H^2(\Phi_2) \ar[r] & \dots \\
E^{p,q}_1:& H^1(\Phi_0) \ar[r] & H^1(\Phi_1) \ar[r] & H^1(\Phi_2) \ar[r] & \dots \\
& H^0(\Phi_0) \ar[r] & H^0(\Phi_1) \ar[r] & H^0(\Phi_2) \ar[r] & \dots \\ 
} 
\end{align}
Therefore, if we suppose that $H$ is $2$-connected and $\G$ is 
$1$-connected, $\Phi$ induces isomorphisms between the total cohomologies 
in all degrees less than or equal to $2$. 

The general case is not much more complicated. Indeed, when $W\neq(0)$, 
the restriction 
\begin{eqnarray}\label{Phi_p}
\xymatrix{
\Phi_p:C^{p,\bullet}_\bullet(\G,\phi) \ar[r] & C^{p,\bullet}_\bullet(\gg_1,\phi)
}
\end{eqnarray}
is a map of honest double complexes for which one can apply the above 
reasoning. The cohomology of $\Phi$ coincides with the total cohomology of 
a double complex of sorts each of whose columns is the total complex of 
$\Phi_p$ (see~(\ref{Cllpsed}) below). In spite of not being an actual double 
complex, this object can be filtrated by columns and the first 
page of its spectral sequence essentially coincides with (\ref{1stPage}); 
ultimately, allowing us to prove:

\begin{theorem}\label{2vanEst}
Let $\G$ be a Lie $2$-group with associated crossed module 
$\xymatrix{G \ar[r] & H}$, Lie $2$-algebra $\gg_1$ and a 
representation on the $2$-vector space $\xymatrix{W \ar[r]^\phi & V}$. 
If $H$ and $G$ are both $2$-connected, the map that extends the 
linearization of extensions induces isomorphisms
\begin{eqnarray}\label{2vanEstIso}
\xymatrix{
\Phi:H_{\nabla}^n(\G,\phi) \ar[r] & H_{\nabla}^n(\gg_1,\phi)
}
\end{eqnarray}
for all $n\leq 2$ and it is injective for $n=3$.
\end{theorem}

We can now run van Est strategy as the final ingredient, that linear Lie 
$2$-algebras are integrable is proved in \cite{Sheng_Zhu2:2012}, and one 
can choose such an integration $\xymatrix{G \ar[r] & H}$ with both $H$ and 
$G$ $1$-connected. Since $1$-connected Lie groups are automatically 
$2$-connected, we conclude that there is a unique Lie $2$-group extension 
integrating the canonical adjoint extension of a Lie $2$-algebra, thus 
implying its integrability.

The body of the article is dedicated to formalize how to apply the van Est 
strategy to the case of strict Lie $2$-algebras. In Section~\ref{sec-Pre}, 
we convene notation and state some basic facts. In Section~\ref{sec-TheMap}, 
we define the van Est map and prove it is a map of complexes. In 
Section~\ref{sec-Theos}, we prove a van Est type theorem realizing the 
van Est map as a composition of groupoid van Est maps.

\section{Preliminaries}\label{sec-Pre}
In this section, we settle notation by reviewing the necessary notions 
to formally define the complexes of Lie $2$-algebra and Lie $2$-group 
cochains with $2$-vector space coefficients, and recall the elements of 
homological algebra of which we make use below. 

\subsection{Lie 2-algebras, Lie 2-groups and their cohomology}\label{subsec-2Coh}
\subsubsection{The equivalence with the categories of crossed modules}\label{sss-Equiv}
In the sequel, we make no distinction between a Lie $2$-group or 
a Lie $2$-algebra and their corresponding crossed modules. For future 
reference, we make explicit the equivalence at the level of objects.  

Let $\G$ be a Lie $2$-group (\ref{ALie2Gp}). In order to make clear the 
difference between the group operation and the groupoid operation in $\G$, 
we assume the following convention: 
\begin{eqnarray*}
\gamma_1\vJoin\gamma_2, & \gamma_3\Join\gamma_4
\end{eqnarray*}
stand respectively for the group multiplication and the groupoid 
multiplication whenever $(\gamma_1,\gamma_2)\in\G^2$ and 
$(\gamma_3,\gamma_4)\in\G\times_H\G$. This notation intends to reflect that 
we think of the group multiplication as being ``vertical'', and the 
groupoid multiplication as being ``horizontal''.

Since the source map is a surjective submersion and the unit map is a 
canonical splitting thereof, letting $G$ be the Lie subgroup 
$\ker s\leq\G$, $\G\cong G\times H$. The associated crossed module map is 
given by $\xymatrix{G \ar[r]^{t\vert_{\ker s}} & H}$. Thinking of the 
underlying $2$-vector space of a Lie $2$-algebra (\ref{ALie2Alg}) as an 
abelian Lie $2$-group, this construction yields a canonical isomorphism 
$\gg_1\cong\gg\oplus\hh$, where $\gg:=\ker\hat{s}$ and the crossed module 
map $\xymatrix{\gg \ar[r]^{\hat{t}\vert_{\gg}} & \hh}$.

In the Lie $2$-group case, the right action is given by conjugation by 
units in the group $\G$
\begin{align}\label{GpAction}
g^h :=u(h)^{-1}\vJoin g\vJoin u(h),
\end{align}  
for $h\in H$ and $g\in G$. We stress that the $-1$ power stands for the 
inverse of the group multiplication $\vJoin$. In the Lie $2$-algebra case, 
the action by derivations is given by 
\begin{align}\label{AlgAction}
\Lie_yx :=[u(y),x]_1,
\end{align} 
for $y\in\hh$ and $x\in\gg$. Here, $[\cdot,\cdot]_1$ stands for the Lie 
bracket of $\gg_1$. 

Conversely, given a crossed module $\xymatrix{G \ar[r]^i & H}$ as in 
Definition~\ref{GpCrossMod}, the space of arrows of its associated Lie 
$2$-group $\G$ is defined to be the product $G\times H$ together with 
the structural maps 
\begin{align*}
s(g,h)=h, & \qquad t(g,h)=hi(g), & \iota(g,h)=(g^{-1},hi(g)), & \qquad u(h)=(1,h), 
\end{align*}
\begin{align}\label{GpStrMaps}
(g',hi(g))\Join(g,h):=(gg',h) 
\end{align} 
for $h\in H$ and $g,g'\in G$. Thinking of the underlying $2$-term 
complex of vector spaces of a crossed module 
$\xymatrix{\gg \ar[r]^\mu & \hh}$ as in Definition~\ref{AlgCrossMod} as an 
abelian crossed module of Lie groups, this construction yields a 
$2$-vector space $\gg_1:=\gg\oplus\hh$ with structural maps given by the 
same formulae that we transcribe using additive notation 
\begin{align*}
\hat{s}(x,y)=y, & \qquad \hat{t}(x,y)=y+\mu(x), & \hat{\iota}(x,y)=(-x,y+\mu(x)), & \qquad \hat{u}(y)=(0,y),            
\end{align*}
\begin{equation}\label{AlgStrMaps}
(x',y+\mu(x))\Join(x,y)=\hat{m}(x',y+\mu(x);x,y):=(x+x',y)
\end{equation}
for $y\in\hh$ and $x,x'\in\gg$. 

In the Lie group case, $\G$ is endowed with the group structure of the  
semi-direct product $G\rtimes H$ with respect to the $H$-action. 
Explicitly, the product is given by 
\begin{align}\label{ArrowsProduct}
(g_1,h_1)\vJoin(g_2,h_2)=(g_1^{h_2}g_2,h_1h_2),
\end{align}
for $(g_1,h_1),(g_2,h_2)\in G\rtimes H$. In the Lie algebra case, $\gg_1$ 
is endowed with the bracket of the semi-direct sum $\gg\oplus_\Lie\hh$, 
explicitly given by  
\begin{align}\label{ArrowsBracket}
[(x_0,y_0),(x_1,y_1)]_\Lie:=([x_0,x_1]+\Lie_{y_0}x_1-\Lie_{y_1}x_0,[y_0,y_1]),
\end{align}
for $(x_0,y_0),(x_1,y_1)\in\gg\oplus\hh$. 

\subsubsection{The general linear Lie 2-group and the linear Lie 2-algebra}\label{sss-LinAndRep}
The General Linear Lie $2$-group $GL(\phi)$ \cite{Norrie:1987,Sheng_Zhu2:2012} 
is the Lie $2$-group which plays the r\^ole of space of automorphisms of 
the $2$-vector space $\xymatrix{W\ar[r]^\phi & V}$. $GL(\phi)$ is by
definition the crossed module
\begin{align}\label{GLinCrossMod}
\xymatrix{
\Delta:GL(\phi)_1 \ar[r] & GL(\phi)_0:A \ar@{|->}[r] & (I+A\phi ,I+\phi A),
}
\end{align}
where
\begin{align*}
GL(\phi)_1 & =\lbrace A\in Hom(V,W):(I+A\phi ,I+\phi A)\in GL(W)\times GL(V)\rbrace, \\
GL(\phi)_0 & =\lbrace(F,f)\in GL(W)\times GL(V):\phi\circ F=f\circ\phi\rbrace, 
\end{align*}
and the right action given by
\begin{align}\label{act}
A^{(F,f)} & :=F^{-1}Af, & \textnormal{for }(F,f)\in GL(\phi)_0,\quad A\in GL(\phi)_1.
\end{align}
The group structure on the Whitehead group $GL(\phi)_1$ is 
given by
\begin{align}\label{gpStr}
A_1\odot A_2 & := A_1+A_2+A_1\phi A_2, & A_1,A_2\in GL(\phi)_1,
\end{align} 
whose identity element is the $0$ map, and where the inverse of an element 
$A$ is given by either $-A(I+\phi A)^{-1}=-(I+A\phi )^{-1}A$.

The Lie $2$-algebra of $GL(\phi)$ is $\ggl(\phi)$, explicitly given 
by \cite{Sheng_Zhu2:2012}
\begin{align}\label{LinCrossMod}
\xymatrix{
\Delta' :\ggl(\phi)_1 \ar[r] & \ggl(\phi)_0:A \ar@{|->}[r] & (A\phi,\phi A),
}
\end{align} 
where
\begin{align*}
\ggl(\phi)_1 & =Hom(V,W), & \ggl(\phi)_0 & =\lbrace (F,f)\in\ggl(W)\oplus\ggl(V):\phi\circ F=f\circ\phi\rbrace, 
\end{align*}
and the action is given by
\begin{align}\label{linAct}
\Lie^{\phi}_{(F,f)}A & :=FA-Af, & \textnormal{for }(F,f)\in\ggl(\phi)_0,\quad A\in\ggl(\phi)_1.
\end{align}
The Lie bracket on $\ggl(\phi)_1$ is given by 
\begin{align}\label{linBrack}
[A_1,A_2]_\phi & := A_1\phi A_2-A_2\phi A_1, & A_1,A_2\in\ggl(\phi)_1.
\end{align}

Regarding the maps (\ref{GLinCrossMod}) and (\ref{LinCrossMod}) as diagonal 
homomorphisms to $GL(W\oplus V)$ and $\ggl(W\oplus V)$, one can deduce the 
formula for the exponential  
$\xymatrix{\exp_{GL(\phi)_1} :\ggl(\phi)_1 \ar[r] & GL(\phi)_1}$, 
\begin{align}\label{TheExpOfGL(phi)}
\exp_{GL(\phi)_1}(A)=A\sum_{n=0}^\infty\frac{(\phi A)^n}{(n+1)!}=\sum_{n=0}^\infty\frac{(A\phi)^n}{(n+1)!}A.
\end{align}

A representation of a Lie $2$-group $\G$ with crossed module 
$\xymatrix{G \ar[r]^i & H}$ on $\xymatrix{W \ar[r]^\phi & V}$ is a morphism 
of Lie $2$-groups 
\begin{align}\label{GpRep}
\xymatrix{
\rho:\G \ar[r] & GL(\phi).
}
\end{align}
Explicitly, $\rho$ consists of Lie group representations of $H$ on $W$ and 
on $V$, 
\begin{align}\label{GpRho0}
\xymatrix{
\rho_0:H \ar[r] & GL(\phi)_0\leq GL(W)\times GL(V):h \ar@{|->}[r] & (\rho_0^1(h),\rho_0^0(h)), 
}
\end{align}
intertwining $\phi$, \textit{i.e.}, such that 
$\phi\circ\rho_0^1(h)=\rho_0^0(h)\circ\phi$ for all $h\in H$; and a Lie 
group homomorphism
\begin{align}\label{GpRho1}
\xymatrix{
\rho_1:G \ar[r] & GL(\phi)_1, 
}
\end{align}
so that
\begin{align}\label{GpRho1Homo}
\rho_1(g_0g_1) & =\rho_1(g_0)+\rho_1(g_1)+\rho_1(g_0)\circ\phi\circ\rho_1(g_1), & \textnormal{for all }g_0,g_1\in G.
\end{align}
Due to the compatibility with the crossed module structure, the following 
relations hold for all $g\in G$ and $h\in H$:
\begin{align}\label{GpRepEqns}
\rho^0_0(i(g)) & =I+\phi\circ\rho_1(g), & \rho^1_0(i(g)) & =I+\rho_1(g)\circ\phi, & \rho_1(g^h) & =\rho_0^1(h)^{-1}\rho_1(g)\rho_0^0(h).
\end{align}

A representation of a Lie $2$-algebra $\gg_1$ with crossed module 
$\xymatrix{\gg \ar[r]^\mu & \hh}$ on $\xymatrix{W \ar[r]^\phi & V}$ is a 
morphism of Lie $2$-algebras 
\begin{align}\label{AlgRep}
\xymatrix{
\rho:\gg \ar[r] & \ggl(\phi).
}
\end{align}
Explicitly, $\rho$ consists of Lie algebra representations of $\hh$ on $W$ 
and on $V$, 
\begin{align}\label{AlgRho0}
\xymatrix{
\rho_0:\hh \ar[r] & \ggl(\phi)_0\leq\ggl(W)\oplus\ggl(V):y \ar@{|->}[r] & (\rho_0^1(y),\rho_0^0(y)), 
}
\end{align}
intertwining $\phi$, \textit{i.e.}, such that 
$\phi\circ\rho_0^1(y)=\rho_0^0(y)\circ\phi$ for all $y\in\hh$; and a Lie 
algebra homomorphism
\begin{align}\label{AlgRho1}
\xymatrix{
\rho_1:\gg \ar[r] & \ggl(\phi)_1. 
}
\end{align}
Due to the compatibility with the crossed module structure, the following 
relations hold for all $x\in\gg$ and $y\in\hh$:
\begin{align}\label{AlgRepEqns}
\rho^0_0(\mu(x)) & =\phi\circ\rho_1(x), & \rho^1_0(\mu(x)) & =\rho_1(x)\circ\phi, & \rho_1(\Lie_yx) & =\rho_0^1(y)\rho_1(x)-\rho_1(x)\rho_0^0(y),
\end{align}
where $\Lie$ is the action of $\hh$ on $\gg$. 

Clearly, the differential of a Lie $2$-group representation $\rho$ yields a 
representation $\dot{\rho}$ of its Lie $2$-algebra. 

\begin{example}\label{2ad}[The adjoint representation] 
Define the adjoint representation of a Lie $2$-algebra on itself by:
\begin{align*}
\ad_1:   & \xymatrix{\gg \ar[r] & \ggl(\mu)_1} & \ad_1(x)(u)   & :=-\Lie_u x ,\quad x\in\gg, u\in\hh \\
\ad_0^1: & \xymatrix{\hh \ar[r] & \ggl(\gg)}   & \ad_0^1(y)(v) & :=\Lie_y v ,\quad y\in\hh, v\in\gg \\
\ad_0^0: & \xymatrix{\hh \ar[r] & \ggl(\hh)}   & \ad_0^0(y)(u) & :=[y,u], \quad y,u\in\hh.    
\end{align*}
\end{example}

\subsubsection{The cochain complexes with coefficients}\label{sss-Cxs}
We define the differentials $\nabla$ for the complexes in (\ref{Cxs}). As 
it was briefly stated, both differentials are defined as a graded sum of 
the differentials in a three-dimensional grid whose vertices are the spaces 
in (\ref{Alg3dimLat}) and (\ref{Gp3dimLat}) together with 
\textit{difference maps} that account for the fact that not all these 
commute.

Let $\G$ be a Lie $2$-group (\ref{ALie2Gp}) with associated crossed module 
$\xymatrix{G \ar[r]^i & H}$. The face maps of the simplicial structure on 
the nerve of $\G$ are given by 
\begin{align}\label{faceMaps}
\partial _k (\gamma_0,...,\gamma_p)=\begin{cases}
(\gamma_1,...,\gamma_p)& \textnormal{if }k=0 \\
(\gamma_0,...,\gamma_{k-2},\gamma_{k-1}\Join\gamma_k,\gamma_{k+1},...,\gamma_p)& \textnormal{if }0<k\leq p \\
(\gamma_0,...,\gamma_{p-1})& \textnormal{if }k=p+1,\end{cases}
\end{align} 
for a given element $(\gamma_0,...,\gamma_p)\in\G_{p+1}$. Under the 
canonical isomorphism $\G\cong G\times H$ (see Subsection~\ref{sss-Equiv}), 
the space of $p$-composable arrows $\G_p$ corresponds to $G^p\times H$; 
hereafter, we consider this isomorphism to be fixed and treat it as an 
equality when necessary. For each coordinate $\gamma_j$ of 
$\gamma=(\gamma_1,...,\gamma_p)\in\G_p\leq\G^p$, there is a corresponding 
$(g_j,h_j)\in G\times H$. According to Eq. (\ref{GpStrMaps}), the defining 
relation for $\G_p$ then reads $h_j=h_{j+1}i(g_{j+1})$, thus making the 
map $\xymatrix{\gamma \ar@{|->}[r] & (g_1,...,g_p;h_p)}$ an isomorphism 
with inverse 
$\xymatrix{(g_1,...,g_p;h) \ar@{|->}[r] & (g_1,hi(g_p...g_2);...;g_{p-1},hi(g_p);g_p,h)}$. 
Under this isomorphism, the face maps become
\begin{align}\label{2GpFaceMaps}
\partial_k(g_0,...,g_p;h)=\begin{cases}
(g_1,...,g_p;h) & \textnormal{if }k=0 \\
(g_0,...,g_{k-2},g_kg_{k-1},g_{k+1},...,g_p;h) & \textnormal{if }0<k\leq p \\
(g_0,...,g_{p-1};hi(g_p))& \textnormal{if }k=p+1.\end{cases}
\end{align}
Let 
$\xymatrix{t_p:\G_p \ar[r] & H:(\gamma_1,...\gamma_p) \ar@{|->}[r] & t(\gamma_1)}$ 
be the map that returns the final target of a $p$-tuple of composable 
arrows. We remark that $t_p$ is a composition of face maps and hence a 
group homomorphism, and rewrite it as
\begin{align}
\xymatrix{
t_p:G^p\times H \ar[r] & H:(g_1,...,g_p;h) \ar@{|->}[r] & hi\Big{(}\prod_{j=0}^{p-1}g_{p-j}\Big{)}=hi(g_p...g_1).
}
\end{align}
If $E$ is a vector bundle over $H$ and 
$\xymatrix{\G{}_s\times_H E \ar[r] & E:(\gamma,e) \ar@{|->}[r] & \Delta_\gamma e}$ 
is a left action along the projection, the differential of the complex of 
Lie groupoid cochains with values on $E$, 
$\xymatrix{\partial:C^\bullet(\G;E)\ar[r] & C^{\bullet+1}(\G;E)}$, where 
$C^p(\G;E):=\Gamma(t_p^*E)$, is defined by 
\begin{align}\label{LGpdDifferential}
(\partial\varphi)(\gamma_0,...,\gamma_p) & =\Delta_{\gamma_0}\partial_0^*\varphi(\gamma_0,...,\gamma_p)+\sum_{k=1}^{p+1}(-1)^k\partial_k^*\varphi(\gamma_0,...,\gamma_p) 
\end{align} 
for $\varphi\in C^p(\G;E)$ and $(\gamma_0,...,\gamma_p)\in\G_{p+1}$. 
Analogously, when $E$ is a right representation, one uses the initial 
source map to define $C^p(\G;E):=\Gamma(s_p^*E)$, and, keeping the notation 
for the action, the differential for $\varphi\in C^p(\G;E)$ and 
$(\gamma_0,...,\gamma_p)\in\G_{p+1}$ is defined to be
\begin{align}\label{RGpdDifferential}
(\partial\varphi)(\gamma_0,...,\gamma_p) & =\sum_{k=0}^{p}(-1)^k\partial_k^*\varphi(\gamma_0,...,\gamma_p)+(-1)^{p+1}\Delta_{\gamma_p}\partial_{p+1}^*\varphi(\gamma_0,...,\gamma_p). 
\end{align}
The differential of the Chevalley-Eilenberg complex of a Lie algebra $\gg$ 
with values in a representation $\rho$ on the vector space $V$, 
\begin{align*}
\xymatrix{
\delta_{CE}:\bigwedge^\bullet\gg^*\otimes V \ar[r] & \bigwedge^{\bullet +1}\gg^*\otimes V,
}
\end{align*}
is defined by 
\begin{align}\label{CEDiff}
(\delta_{CE}\omega)(X) & =\sum_{j=0}^q(-1)^j\rho(x_j)\omega(X(j))+\sum_{m<n}(-1)^{m+n}\omega([x_m,x_n],X(m,n))
\end{align}
for $\omega\in\bigwedge^q\gg^*\otimes V$ and $X=(x_0,...,x_q)\in\gg^{q+1}$. 
We adopt the convention that, for $0\leq a_1<...<a_k\leq q$, 
\begin{align}\label{conv}
 X(a_1,...,a_k) & :=(x_0,...,x_{a_1-1},x_{a_1+1},...,x_{a_k-1},x_{a_k+1},...,x_{q}),
\end{align}
as opposed to the usual $\hat{\cdot}$ notation.

To define the complex of Lie $2$-algebra cochains with values on 
$\xymatrix{W \ar[r]^\phi & V}$, fix a Lie $2$-algebra (\ref{ALie2Alg}) 
with associated crossed module $\xymatrix{\gg \ar[r]^\mu & \hh}$ whose 
action we write $\Lie$ and fix a representation $\rho$ (\ref{AlgRep}). 
We think of $\rho$ as a triple $(\rho_0^1,\rho_0^0;\rho_1)$, where 
$\rho_0=(\rho_0^1,\rho_0^0)$ is the map (\ref{AlgRho0}) and $\rho_1$ is 
the map (\ref{AlgRho1}). Let $C^{p,q}_r(\gg_1,\phi)$ be given by
(\ref{Alg3dimLat}). 

{\bf The p-direction} - For constant $r$, when $q=0$, one has got the trivial complexes
\begin{align*}
\xymatrix{
C^{0,0}_r(\gg_1,\phi) \ar[r]^{\partial=0} & C^{1,0}_r(\gg_1,\phi) \ar[r]^{\partial=Id} & C^{2,0}_r(\gg_1,\phi) \ar[r]^{\partial=0} & C^{3,0}_r(\gg_1,\phi) \ar[r] & \cdots ;
}
\end{align*}
for $q>0$, we set the complexes
\begin{eqnarray*}
\xymatrix{
\partial:\bigwedge^q\gg_\bullet^*\otimes V \ar[r] & \bigwedge^q\gg_{\bullet+1}^*\otimes V
} & \textnormal{and} & \xymatrix{
\partial:\bigwedge^q\gg_\bullet^*\otimes\bigwedge^r\gg^*\otimes W \ar[r] & \bigwedge^q\gg_{\bullet+1}^*\otimes\bigwedge^r\gg^*\otimes W
}
\end{eqnarray*}
to be the subcomplexes of multilinear alternating groupoid cochains of 
the complex of $\xymatrix{\gg _1^q \ar@<0.5ex>[r] \ar@<-0.5ex>[r] & \hh^q}$ 
with values in the trivial representation on 
$\xymatrix{\hh^q\times V \ar[r] & \hh^q}$ and 
$\xymatrix{\hh^q\times(\bigwedge^r\gg^*\otimes W) \ar[r] & \hh^q}$ 
respectively. 

{\bf The q-direction} - For constant $p$ and $r=0$, we set the complex 
\begin{align*}
\xymatrix{
\delta:\bigwedge^\bullet\gg_p^*\otimes V \ar[r] & \bigwedge^{\bullet+1}\gg_p^*\otimes V
}
\end{align*}
to be the Chevalley-Eilenberg complex of $\gg_p$ with values in 
the pull-back representation $\rho_p:=\hat{t}_p^*\rho_0^0$ on $V$. 
For constant $r>0$, we set the complex 
\begin{align*}
\xymatrix{
\delta:\bigwedge^\bullet\gg_p^*\otimes\bigwedge^r\gg^*\otimes W \ar[r] & \bigwedge^{\bullet+1}\gg_p^*\otimes\bigwedge^r\gg^*\otimes W
}
\end{align*}
to be the Chevalley-Eilenberg complex of $\gg_p$ with values in 
the pull-back representation $\rho_p^{(r)}:=\hat{t}_p^*\rho^{(r)}$ on 
$\bigwedge^r\gg^*\otimes W$, where 
$\xymatrix{\rho^{(r)}:\hh \ar[r] & \ggl(\bigwedge ^r\gg ^*\otimes W)}$ is 
the dual representation given for $\omega\in\bigwedge^r\gg^*\otimes W$, 
$y\in\hh$ and $z_1,...,z_r\in\gg$ by
\begin{align}\label{q-rep}
\rho ^{(r)}(y)\omega(z_1,...,z_r) & :=\rho_0^1(y)\omega(z_1,...,z_r)-\sum_{k=1}^r\omega(z_1,...,\Lie _y z_k,...,z_r).
\end{align} 

{\bf The r-direction} - For constant $p$ and $q$, we set the complex
\begin{align*}
\xymatrix{
\delta_{(1)}:\bigwedge^q\gg_p^*\otimes\bigwedge^{\bullet}\gg ^*\otimes W \ar[r] & \bigwedge^q\gg_p^*\otimes\bigwedge ^{\bullet+1}\gg^*\otimes W
}
\end{align*}
to be the Chevalley-Eilenberg complex of $\gg$ with values in 
$\xymatrix{\rho_{(1)}:\gg \ar[r] & \ggl(\bigwedge^q\gg _p^*\otimes W)}$ 
given for $\omega\in\bigwedge^q\gg _p^*\otimes W$, $\Xi\in\gg_p^q$ and 
$z\in\gg$ by
\begin{eqnarray}\label{rRep}
\rho_{(1)}(z)\omega(\Xi):=\rho_0^1(\mu(z))\omega(\Xi).
\end{eqnarray}
Since the $0$th degree is $\bigwedge^q\gg_p^*\otimes V$ instead of 
$\bigwedge^q\gg_p^*\otimes W$, we specify the first differential to be
\begin{align}\label{Alg1st r}
\xymatrix{
\delta':\bigwedge^q\gg_p^*\otimes V \ar[r] & \bigwedge^q\gg_p^*\otimes\gg^*\otimes W
} & :\delta'\omega(\Xi;z)=\rho_1(z)\omega(\Xi),
\end{align}
where $\Xi\in\gg_p^q$ and $z\in\gg$. 

{\bf Difference maps} -  The $k$th difference map  
\begin{align*}
\xymatrix{
\Delta_{k}:C^{p,q}_r(\gg_1,\phi) \ar[r] & C^{p+1,q+k}_{r-k}(\gg_1,\phi)
}
\end{align*}
is defined for $\omega\in C^{p,q}_r(\gg_1,\phi)$, $Z\in\gg^{r-k}$ and 
$\Xi=(\xi_0,...,\xi_{q+k-1})\in\gg_{p+1}^{q+k}$ by
\begin{align}\label{AlgHigherDiffMaps}
(\Delta_k\omega)(\Xi;Z) & :=\sum_{a_1<...<a_k}(-1)^{a_1+...+a_k}\omega(\hat{\partial}_0\Xi(a_1,...,a_k);x_{a_1}^0,...,x_{a_k}^0,Z).
\end{align}
Here, we used the notation convention (\ref{conv}), and identified each 
$\xi_j$ with $j\in\lbrace0,...,q+k-1\rbrace$ with $(x_j^0,...,x_j^p;y_j)$. 
In the special case $k=r$, the map $\Delta_r$ is essentially defined by 
Eq.~(\ref{AlgHigherDiffMaps}), but composed with $\phi$, so that 
it takes values in the right vector space. We sometimes drop the subindex 
of the first difference map and write $\Delta$ instead of $\Delta_1$.

The three dimensional grid knit by $\partial$, $\delta$ and $\delta_{(1)}$ 
falls short of defining a triple complex because $\partial$ and $\delta$ do 
not commute. Since all difference maps are homogeneous of degree $+1$ with 
respect to the diagonal grading, it makes sense to use them as 
differentials: 
\begin{align}\label{AlgNabla}
\nabla & :=\partial+(-1)^{p+q}(\delta+\delta_{(1)})+\sum_{k=1}^r\Delta_k .
\end{align}
The higher difference maps in Eq.~(\ref{AlgNabla}) make up for the 
non-commuting differentials, so that $\nabla$ squares to zero 
\cite{Angulo1:2020}. Thus defined, the complex 
$(C_\nabla(\gg_1,\phi),\nabla)$ verifies the following property.
\begin{theorem}\cite{Angulo1:2020}\label{H2Alg}
$H^2_\nabla(\gg_1,\phi)$ is in one-to-one correspondence with 
isomorphism classes of extensions of the Lie $2$-algebra 
$\gg_1$ by the $2$-vector space $\xymatrix{W\ar[r]^\phi & V.}$
\end{theorem}

To define the complex of Lie $2$-group cochains with values on 
$\xymatrix{W \ar[r]^\phi & V}$, fix a Lie $2$-group (\ref{ALie2Gp}) 
with associated crossed module $\xymatrix{G \ar[r]^i & H}$ for whose 
action we use exponential notation and fix a representation $\rho$ 
(\ref{GpRep}). We think of $\rho$ as a triple $(\rho_0^1,\rho_0^0;\rho_1)$, 
where $\rho_0=(\rho_0^1,\rho_0^0)$ is the map (\ref{GpRho0}) and $\rho_1$ 
is the map (\ref{GpRho1}). Let $C^{p,q}_r(\G,\phi)$ be given by
(\ref{Gp3dimLat}). 

{\bf The p-direction} - For constant $r$, when $q=0$, one has got the trivial 
complexes
\begin{align*}
\xymatrix{
C^{0,0}_r(\G,\phi) \ar[r]^{\partial=0} & C^{1,0}_r(\G,\phi) \ar[r]^{\partial=Id} & C^{2,0}_r(\G,\phi) \ar[r]^{\partial=0} & C^{3,0}_r(\G,\phi) \ar[r] & \cdots 
}
\end{align*}
When $q>0$ and $r=0$, we set the complex
\begin{align*}
\xymatrix{
\partial:C(\G_{\bullet}^q,V) \ar[r] & C(\G_{\bullet+1}^q,V)
}
\end{align*}
to be the cochain complex of the product groupoid 
$\xymatrix{\G^q \ar@<0.5ex>[r]\ar@<-0.5ex>[r] & H^q}$ with respect to the 
trivial representation on the vector bundle 
$\xymatrix{H^q\times V \ar[r] & H^q}$. For any other value of $r$, we 
set the complex
\begin{align*}
\xymatrix{
\partial:C(\G_{\bullet}^q\times G^r,W) \ar[r] & C(\G_{\bullet+1}^q\times G^r,W)
}
\end{align*}
to be the cochain complex of the product groupoid 
$\xymatrix{\G^q\times G^r \ar@<0.5ex>[r]\ar@<-0.5ex>[r] & H^q\times G^r}$ 
with respect to the left representation on the trivial bundle 
$\xymatrix{H^q\times G^r\times W \ar[r] & H^q\times G^r}$ given for 
$\gamma_k=(g_k,h_k)\in\G$ and $\vec{f}\in G^r$ by
\begin{align}\label{p-rep}
(\gamma_1,...,\gamma_q;\vec{f})\cdot(h_1,...,h_q;\vec{f},w) & :=(h_1i(g_1),...,h_qi(g_q);\vec{f},\rho_0^1(i(pr_G(\gamma_1\vJoin\cdots\vJoin\gamma_q)))^{-1}w). 
\end{align}

{\bf The q-direction} - When $r=0$, we set the complex 
\begin{align*}
\xymatrix{
\delta:C(\G_p^{\bullet},V) \ar[r] & C(\G_p^{\bullet+1},V)
}
\end{align*}
to be the group complex of $\G_p$ with values in the pull-back of the 
representation $\rho_0^0$ along the final target map $t_p$; when 
$r\neq 0$, we set the complex 
\begin{align*}
\xymatrix{
\delta:C(\G_p^{\bullet}\times G^r,W) \ar[r] & C(\G_p^{\bullet+1}\times G^r,W)
}
\end{align*}
to be the cochain complex of the (right!) transformation groupoid 
$\xymatrix{\G_p\ltimes G^r \ar@<0.5ex>[r]\ar@<-0.5ex>[r] & G^r}$ with 
respect to the right representation
\begin{align}\label{q-Rep}
(g_1,...,g_r;w)\cdot(\gamma;g_1,...,g_r) & :=(g_1^{t_p(\gamma)},...,g_r^{t_p(\gamma)};\rho_0^1(t_p(\gamma))^{-1}w)
\end{align}
on the trivial vector bundle $\xymatrix{G^r\times W \ar[r] & G^r}$, where 
$g_1,...,g_r\in G$, $\gamma\in\G_p$ and $w\in W$. When writing the groupoid 
differential, we use the shorthand $\rho^r_{\G_p}(\gamma;g)w$ instead of 
the lengthier Eq. (\ref{q-Rep}).
 
{\bf The r-direction} - When $q=0$, we set the complex 
\begin{align*}
\xymatrix{
\delta_{(1)}:C(G^{\bullet},W) \ar[r] & C(G^{\bullet+1},W)
}
\end{align*}
to be the group complex of $G$ with values in the pull-back of the 
representation $\rho_0^1$ along the crossed module homomorphism $i$, but 
for the $0$th degree; when $q\neq 0$, we set the complex 
\begin{align}\label{rDiff}
\xymatrix{
\delta_{(1)}:C(\G_p^q\times G^{\bullet},W) \ar[r] & C(\G_p^q\times G^{\bullet+1},W),
}
\end{align}
again except for the $0$th degree, to be the cochain complex of the Lie 
group bundle 
$\xymatrix{\G_p^q\times G \ar@<0.5ex>[r]\ar@<-0.5ex>[r] & \G_p^q}$ with 
respect to the left representation
\begin{align}\label{r-Rep}
(\gamma_1,...,\gamma_q;g)\cdot(\gamma_1,...,\gamma_q;w) & :=(\gamma_1,...,\gamma_q;\rho_0^1(i(g^{t_p(\gamma_1)...t_p(\gamma_q)}))w)
\end{align}
on the trivial vector bundle $\xymatrix{G_p^q\times W \ar[r] & G_p^q}$, 
where $\gamma_1,...,\gamma_q\in\G_p$, $g\in G_p$ and $w\in W$.  
Though right and left representations of a Lie group bundle coincide, we 
emphasize that Eq. (\ref{r-Rep}) be taken as a left representation. 

The missing maps $\xymatrix{\delta':V \ar[r] & C(G,W)}$ and 
$\xymatrix{\delta':C(\G_p^q,V) \ar[r] & C(\G_p^q\times G,W)}$ are defined 
respectively for $v\in V$, $g\in G$, $\omega\in C(\G_p^q,V)$ and 
$\gamma_1,...,\gamma_q\in\G_p$ by
\begin{eqnarray}\label{Gp1st r}
(\delta'v)(g):=\rho_1(g)v & \textnormal{ and } & (\delta'\omega)(\gamma_1,...,\gamma_q;g)=\rho_0^1(t_p(\gamma_1)...t_p(\gamma_q))^{-1}\rho_1(g)\omega(\gamma_1,...,\gamma_q).
\end{eqnarray}

{\bf Difference maps} - In order to define the first difference maps, we 
introduce the following notation. We think of an element 
$\vec{\gamma}\in\G_p^q$ as having components
\begin{equation}\label{gammaMatrix}
\vec{\gamma}=\begin{pmatrix}
\gamma_1 \\ \vdots \\ \gamma_q
\end{pmatrix}=\begin{pmatrix}
    \gamma_{11} & \gamma_{12} & ... & \gamma_{1p} \\
    \gamma_{21} & \gamma_{22} & ... & \gamma_{2p} \\
   \vdots  & \vdots &     & \vdots \\
    \gamma_{q1} & \gamma_{q2} & ... & \gamma_{qp} \end{pmatrix}=\begin{pmatrix}
    g_{11} & g_{12} & ... & g_{1p} & h_1 \\
    g_{21} & g_{22} & ... & g_{2p} & h_2 \\
   \vdots  & \vdots &     & \vdots & \vdots \\
    g_{q1} & g_{q2} & ... & g_{qp} & h_q \end{pmatrix},
\end{equation}
where the last equality is a notation abuse corresponding to the 
row-wise isomorphism $\G_{p}\cong G^{p}\times H$. Here, for each value of 
$a$ and $b$, $(g_{ab},h_{ab})$ is the image of $\gamma_{ab}$ under the 
canonical isomorphism $\G\cong G\rtimes H$ (see Subsection~\ref{sss-Equiv}). 

We proceed to define the difference maps 
\begin{align*}
\xymatrix{
\Delta_{a,b}:C^{p,q}_r(\G,\phi) \ar[r] & C^{p+a,q+b}_{r+1-(a+b)}(\G,\phi)
}
\end{align*}
for $(a,b)\in\lbrace(1,1),(r,1),(1,r)\rbrace$. If $a+b=r+1$, 
$\omega\in C(\G_p^q\times G^{a+b-1},W)$ and $\vec{\gamma}\in\G_{p+a}^{q+b}$, 
set 
\begin{align}\label{Delta-ab}
(\Delta_{a,b}\omega)(\vec{\gamma}) & :=\rho_0^0(t_p(\partial_0^a\gamma_1)...t_p(\partial_0^a\gamma_{q+b}))\circ\phi\big{(}\omega(\partial_0^a\delta_0^b\vec{\gamma};c_{a,b}(\vec{\gamma}))\big{)},
\end{align}
where $\xymatrix{c_{a,b}:\G_{p+a}^{q+b} \ar[r] & G^r}$ are respectively
\begin{eqnarray}\label{most}
c_{1,1}(\vec{\gamma}):=g_{11}, & c_{r,1}(\vec{\gamma}):=(g_{1r},g_{1r-1},...,g_{12},g_{11}), & \textnormal{and }c_{1,r}(\vec{\gamma}):=(g_{11}^{h_{21}...h_{r1}},g_{21}^{h_{31}...h_{r1}},...,g_{(r-1)1}^{h_{r1}},g_{r1}).
\end{eqnarray}
When $r>1$, 
\begin{align*}
\xymatrix{
\Delta_{1,1}:C(\G_p^q\times G^{r},W) \ar[r] & C(\G_{p+1}^{q+1}\times G^{r-1},W)
}
\end{align*}
is defined for $\omega\in C(\G_p^q\times G^r,W)$, 
$\vec{f}=(f_1,...,f_{r-1})\in G^{r-1}$ and $\vec{\gamma}\in\G_{p+1}^{q+1}$ 
as in Eq. (\ref{gammaMatrix}) by
\begin{align*}
(\Delta_{1,1}\omega)(\vec{\gamma};\vec{f}) =\rho_0^1(i(pr_G(\gamma_{21}\vJoin\cdots\vJoin & \gamma_{(q+1)1})))^{-1}\Big{[}\rho_0^1(i(g_{11}^{h_{21}...h_{(q+1)1}}))^{-1}\omega(\partial_0\delta_0\vec{\gamma};(\vec{f})^{h_{11}},g_{11})+ \\
 & +\sum_{n=1}^{r-1}(-1)^{n+1}\big{(}\omega(\partial_0\delta_0\vec{\gamma};c_{2n-1}(\vec{f};\gamma_{11}))-\omega(\partial_0\delta_0\vec{\gamma};c_{2n}(\vec{f};\gamma_{11}))\big{)}\Big{]}, 
\end{align*}
where $(\vec{f})^{h_{11}}:=(f_1^{h_{11}},...,f_{r-1}^{h_{11}})$ and 
$\xymatrix{c_{2n-1},c_{2n}:G^{r-1}\times \G \ar[r] & G^r}$ are respectively 
given by
\begin{align}\label{c(11)}
c_{2n-1}(\vec{f};\gamma_{11}) & :=\Big{(}f_1^{h_{11}i(g_{11})},...,f_{r-n-1}^{h_{11}i(g_{11})},g_{11}^{-1},f_{r-n}^{h_{11}},...,f_{r-2}^{h_{11}},f_{r-1}^{h_{11}}g_{11}\Big{)}, \\
c_{2n}(\vec{f};\gamma_{11}) & :=\Big{(}f_1^{h_{11}i(g_{11})},...,f_{r-n-1}^{h_{11}i(g_{11})},g_{11}^{-1},f_{r-n}^{h_{11}},...,f_{r-2}^{h_{11}},g_{11}\Big{)},
\end{align}
for $0<n<r$. We often drop the subindex of the difference map 
$\Delta_{1,1}$ and write $\Delta$.

In the Lie $2$-group case, $\partial$ and $\delta$ do not commute either and 
stand in the way of yielding a triple complex. Nonetheless, the difference 
maps are homogeneous of degree $+1$ with respect to the diagonal grading, 
and make up for the non-commuting differentials \cite{Angulo2:2020} by 
setting: 
\begin{align}\label{GpNabla}
\nabla & :=(-1)^p\Big{(}\delta_{(1)}+\sum_{a+b>0}(-1)^{(a+1)(r+b+1)}\Delta_{a,b}\Big{)},
\end{align}
where $\Delta_{1,0}:=\partial$, $\Delta_{0,1}:=\delta$, and 
$\Delta_{a,0}=\Delta_{0,b}=0$ whenever $a,b>1$. Thus defined, the complex 
$(C_\nabla(\G,\phi),\nabla)$ verifies the following property.
\begin{theorem}\cite{Angulo2:2020}\label{H2Gp}
$H^2_\nabla(\G,\phi)$ is in one-to-one correspondence with isomorphism 
classes of split extensions of the Lie $2$-group $\G$ by the $2$-vector 
space $\xymatrix{W \ar[r]^\phi & V}$. 
\end{theorem}

\begin{remark}\label{Incomplete}
As of the writing of this paper, a general formula for $\Delta_{a,b}$ is 
still unavailable. In \cite{Angulo2:2020}, there are formulas for several 
families of difference maps and, in particular, the complex of Lie $2$-group 
cochains with values on a $2$-vector is defined up until degree $5$. Due to 
the scope of our application, we just included the necessary maps to define 
the complex up to degree $2$.
\end{remark}

\subsection{A brief excursus on homological algebra}\label{subsec-HomAlg}
As it was stated in the Introduction, we make use of the van Est theorem 
written in terms of the mapping cone. Though we could not find 
Proposition~\ref{ConeCoh} as stated in the literature, it follows 
from standard techniques that can be found in \cite{Weibel:1994}.

Given a map $\xymatrix{\Phi:(A^\bullet,d_A) \ar[r] & (B^\bullet,d_B)}$, 
the mapping cone of $\Phi$ is defined to be  
\begin{align*}
C(\Phi) & :=(A[1]\oplus B,d_\Phi ),
\end{align*}
where $d_\Phi=\begin{pmatrix}
-d_A & 0 \\
\Phi & d_B 
\end{pmatrix}$. $C(\Phi)$ is complex if and only if $\Phi$ is a map of 
complexes. We write $H(\Phi)$ for the cohomology of the mapping cone of 
$\Phi$.

\begin{proposition}\label{ConeCoh}
Let $\xymatrix{\Phi:A^\bullet \ar[r] & B^\bullet}$ be a map of complexes. 
The following are equivalent:
\begin{itemize}
\item[i)] $H^n(\Phi)=(0)$ for $n\leq k$.
\item[ii)] The induced map in cohomology
\begin{eqnarray*}
\xymatrix{
\Phi^n :H^n(A) \ar[r] & H^n(B),
}
\end{eqnarray*}
is an isomorphism for $n\leq k$, and it is injective for $n=k+1$.
\end{itemize}
\end{proposition}

\begin{proof}
Clearly, $C(\Phi)$ fits in an exact sequence
\begin{align}\label{ExactConeSeq}
\xymatrix{0 \ar[r] & B \ar[r]^j & C(\Phi) \ar[r]^\pi & A[1] \ar[r] & 0, }
\end{align}
whose associated long exact sequence in cohomology is
\begin{eqnarray*}
\xymatrix{
0 \ar[r] & \cancelto{0}{H^{-1}(B)} \ar[r]^{j^{-1}} & H^{-1}(\Phi) \ar[r]^{\pi^{-1}} & H^{-1}(A[1]) \ar[r] & H^0(B) \ar[r]^{j^0} & H^0(\Phi) \ar[r]^{\pi^0} & ...  
}\\
\qquad\qquad\xymatrix{
... \ar[r]^{\pi^0\quad} & H^0(A[1]) \ar[r] & H^1(B) \ar[r]^{j^1} & H^1(\Phi) \ar[r]^{\pi^1} & H^1(A[1]) \ar[r] & ...
}
\end{eqnarray*}
After observing that $H^n(A[1])=H^{n+1}(A)$, and that the connecting 
homomorphism is $\Phi^*$, the proof is straightforward. 
\end{proof}

Proposition~\ref{ConeCoh} holds for the complexes defined by 
Eq.'s~(\ref{AlgNabla}) and (\ref{GpNabla}); however, it will be relevant to 
us that the cone of $\Phi$ can be endowed with a triple grading inherited 
from (\ref{Cxs}). First, suppose 
$\xymatrix{\Phi:A^{\bullet,\bullet} \ar[r] & B^{\bullet,\bullet}}$ is a 
map between double complexes, we define the \textit{mapping cone double} 
$C^{p,q}(\Phi)=A^{p,q+1}\oplus B^{p,q}$ by setting the $p$th column to be 
the mapping cone of $\Phi\vert_{A^{p,\bullet}}$. The product of the 
horizontal differentials defines a map of complexes between columns, and 
thus a double complex, if and only if $\Phi$ is a map of double complexes. 
$C^{\bullet,\bullet}(\Phi)$ fits in an exact sequence of double complexes 
analogous to (\ref{ExactConeSeq}); therefore, the conclusion of 
Proposition~\ref{ConeCoh} holds for its total cohomology $H_{tot}(\Phi)$ 
and the total cohomologies of $A$ and $B$. 

Next, consider a map 
$\xymatrix{\Phi:C_\nabla(\G,\phi) \ar[r] & C_\nabla(\gg_1,\phi)}$ that 
respects the triple grading of the complexes (\ref{Cxs}). We define the 
\textit{mapping cone triple} 
$C^{p,q}_r(\Phi)=C^{p,q+1}_r(\G,\phi)\oplus C^{p,q}_r(\gg_1,\phi)$ by 
setting, for constant $p$, the $p$-page to be the mapping cone double of 
the map $\Phi_p$ (\ref{Phi_p}) and formally summing the remaining 
differentials and difference maps. The total differential $\nabla_\Phi$ 
squares to zero if and only if $\Phi$ is a map of complexes. 
$C^{\bullet,\bullet}_{\bullet}(\Phi)$ fits in an exact sequence analogous 
to (\ref{ExactConeSeq}); therefore, the conclusion of 
Proposition~\ref{ConeCoh} holds for its total cohomology $H_{\nabla}(\Phi)$ 
and the cohomologies of $\G$ and $\gg_1$. 

We point out that the cohomology of the mapping cone triple of $\Phi$ 
tautologically coincides with the total cohomology of 
\begin{align}\label{Cllpsed}
\xymatrix{
\vdots & \vdots & \vdots & \vdots \\ 
C_{tot}^3(\Phi_0) \ar[r]\ar[u]\ar@{-->}[rrd]\ar@{.>}[rrrdd] & C_{tot}^3(\Phi_1) \ar[r]\ar[u]\ar@{-->}[rrd]\ar@{.>}[rrrdd] & C_{tot}^3(\Phi_2) \ar[r]\ar[u]\ar@{-->}[rrd] & C_{tot}^3(\Phi_3) \ar[r]\ar[u] & \dots \\
C_{tot}^2(\Phi_0) \ar[r]\ar[u]\ar@{-->}[rrd] & C_{tot}^2(\Phi_1) \ar[r]\ar[u]\ar@{-->}[rrd] & C_{tot}^2(\Phi_2) \ar[r]\ar[u]\ar@{-->}[rrd] & C_{tot}^2(\Phi_3) \ar[r]\ar[u] & \dots \\
C_{tot}^1(\Phi_0) \ar[r]\ar[u] & C_{tot}^1(\Phi_1) \ar[r]\ar[u] & C_{tot}^1(\Phi_2) \ar[r]\ar[u] & C_{tot}^1(\Phi_3) \ar[r]\ar[u] & \dots \\
C_{tot}^0(\Phi_0) \ar[r]\ar[u] & C_{tot}^0(\Phi_1) \ar[r]\ar[u] & C_{tot}^0(\Phi_2) \ar[r]\ar[u] & C_{tot}^0(\Phi_3) \ar[r]\ar[u] & \dots ,
}
\end{align}
each of whose columns is the total complex of the mapping cone double of 
$\Phi_p$ (\ref{Phi_p}) and where the maps 
$\xymatrix{C_{tot}^q(\Phi_p) \ar[r] & C_{tot}^{q+1-a}(\Phi_{p+a})}$ - 
exemplified respectively for $a\in\lbrace 1,2,3\rbrace$ by the horizontal, 
dashed and pointed arrows in (\ref{Cllpsed})- correspond to the sum 
$\sum_{b=0}^k\Delta_{a,b}$. One may filter (\ref{Cllpsed}) by columns 
giving rise to a spectral sequence whose first page is 
$E^{p,q}_1=H_{tot}^q(\Phi_p)$.

We close this section with the following lemma, which is much used below. 
It follows from a simple application of spectral sequences (see, 
\textit{e.g.}, \cite{Weibel:1994}). 

\begin{lemma}\label{BelowDiag}
Let $(E^{p,q}_r,d_r)$ be the spectral sequence of a double graded object 
$C^{\bullet,\bullet}$ that can be filtrated by columns. If there is a page 
for which $E^{p,q}_r$ is zero for all $(p,q)$ satisfying $p+q\leq k$, then 
\begin{eqnarray*}
H_{tot}^n(C)=(0)\quad\textnormal{for }n\leq k.
\end{eqnarray*}
\end{lemma}

\section{The 2-van Est Map}\label{sec-TheMap}
In this section, we define the van Est map and prove that it defines a map 
of complexes. Throughout, let $\G$ be a Lie $2$-group with Lie $2$-algebra 
$\gg_1$ and let $\rho$ be a representation on $\xymatrix{W \ar[r]^\phi & V}$. 
Define the van Est map
\begin{align}\label{vanEstMap}
\Phi & :\xymatrix{
C^{p,q}_r(\G,\phi) \ar[r] & C^{p,q}_r(\gg_1,\phi),
} & (\Phi\omega)(\xi_1,...,\xi_q ;z_1,...,z_r) & :=\sum_{\sigma\in S_q}\sum_{\varrho\in S_r}\abs{\sigma}\abs{\varrho}\overrightarrow{R}_{\sigma(\Xi)}\overrightarrow{R}_{\varrho(Z)}\omega,
\end{align}
where $\Xi=(\xi_1,...,\xi_q)\in\gg_p^q$, $Z=(z_1,...,z_r)\in\gg^r$, $\abs{\cdot}$ stands 
for the sign of the permutation, and 
\begin{align*}
(\overrightarrow{R}_{\varrho(Z)}\omega)(\vec{\gamma}) & :=\frac{d}{d\tau_r}\rest{\tau_r=0}\cdots\frac{d}{d\tau_1}\rest{\tau_1=0}\omega(\vec{\gamma};\exp_G(\tau_1 z_{\varrho(1)}),...,\exp_G(\tau_r z_{\varrho(r)})), & \textnormal{for } & \vec{\gamma}\in\G_p^q; \\
\overrightarrow{R}_{\sigma(\Xi)}\overrightarrow{R}_{\varrho(Z)}\omega & =\frac{d}{d\lambda_q}\rest{\lambda_q=0}\cdots\frac{d}{d\lambda_1}\rest{\lambda_1=0}(\overrightarrow{R}_{\varrho(Z)}\omega)(\exp_{\G_p}(\lambda_1\xi_{\sigma(1)}),...,\exp_{\G_p}(\lambda_q\xi_{\sigma(q)})).
\end{align*}
Regarding each of these $\overrightarrow{R}$ operators as compositions of 
independent $R_\bullet$ operators that lower degree by differentiating a  
single entry in the direction of the right-invariant vector field of 
$\bullet$ and evaluating at the identity, it is clear that 
$\Phi\omega\in C^{p,q}_r(\gg_1,\phi)$ and $\Phi$ is thus well-defined. 
Observe that, in the page $r=0$, the formula (\ref{vanEstMap}) coincides 
with the classic van Est map sending $\G_p$ cochains to $\gg_p$ cochains;  
analogously, in the page $q=0$, (\ref{vanEstMap}) coincides with the 
classic van Est map for the Lie group $G$.
 
We show that, under the correspondences of Theorems~\ref{H2Alg} and 
\ref{H2Gp}, $\Phi$ is the map that takes in a Lie $2$-group extension 
(or an isomorphism class thereof) and returns its Lie $2$-algebra. 
Let 
$\vec{\omega}=(\varphi,\omega_0,\alpha,\omega_1)\in C^{1,1}_0(\G,\phi)\oplus C^{0,2}_0(\G,\phi)\oplus C^{0,1}_1(\G,\phi)\oplus C^{0,0}_2(\G,\phi)$. 
In the course of proving Theorem~\ref{H2Gp}, one learns that 
$\vec{\omega}$ is a $2$-cocycle if and only if 
\begin{align}\label{CocGpExt}
\xymatrix{G{}_{\rho^1_0\circ i}\ltimes^{\omega_1}W \ar[r] & H{}_{\rho^0_0}\ltimes^{\omega_0}V:(g,w) \ar@{|->}[r] & (i(g),\phi(w)+\varphi{\begin{pmatrix}
g \\
1
\end{pmatrix}})}
\end{align}
with action 
\begin{align}\label{AlphaAction}
(g,w)^{(h,v)} & =(g^h,\rho_0^1(h)^{-1}(w+\rho_1(g)v)+\alpha(h;g)) & \textnormal{for }(g,w)\in G\times W, (h,v)\in H\times V,
\end{align}
is a crossed module. Here, we use the isomorphism of 
Subsection~\ref{sss-Equiv} to cast $\varphi$ as a function of $G\rtimes H$. 
Also, $X{}_{\rho}\ltimes^{\omega}A$ stands for the twisted semi-direct product 
with respect to the representation $\rho$ of $X$ on $A$ defined by 
$\omega\in C_{Gp}^2(X,A)$, which is a Lie group if and 
only if $\omega$ is a Lie group $2$-cocycle. 

If $\vec{\omega}$ is a $2$-cocycle, in particular, so are $\omega_0$ and 
$\omega_1$. Hence, $\Phi\omega_0$ and $\Phi\omega_1$ are Lie algebra 
$2$-cocycles respectively defining the Lie algebras of 
$H{}_{\rho^0_0}\ltimes^{\omega_0}V$ and 
$G{}_{\rho^1_0\circ i}\ltimes^{\omega_1}W$ as the twisted semi-direct sums 
$\hh_{\dot{\rho}^0_0}\oplus^{\Phi\omega_0}V$ and 
$\gg{}_{\dot{\rho}^1_0\circ\mu}\oplus^{\Phi\omega_1}W$. Differentiating 
(\ref{CocGpExt}) at the identity yields
\begin{align}\label{CocAlgExt}
\xymatrix{\gg{}_{\dot{\rho}^1_0\circ\mu}\oplus^{\Phi\omega_1}W \ar[r] & \hh{}_{\dot{\rho}^0_0}\oplus^{\Phi\omega_0}V:(x,w) \ar@{|->}[r] & (\mu(x),\phi(w)+d_{(1,1)}\varphi{\begin{pmatrix}
x \\
0
\end{pmatrix}})}, 
\end{align}
and, clearly, $d_{(1,1)}\varphi=\Phi\varphi$. As for the action, each 
$(h,v)\in H\times V$ defines a Lie group automorphism 
\begin{align}\label{(h,v)-act}
\xymatrix{(-)^{(h,v)}:G{}_{\rho^1_0\circ i}\ltimes^{\omega_1}W \ar[r] & G{}_{\rho^1_0\circ i}\ltimes^{\omega_1}W}.
\end{align}
Differentiating (\ref{(h,v)-act}) at the identity yields a Lie algebra 
automorphism, for which we use the same notation
\begin{align*}
\xymatrix{(-)^{(h,v)}:\gg{}_{\dot{\rho}^1_0\circ\mu}\oplus^{\Phi\omega_1}W \ar[r] & \gg{}_{\dot{\rho}^1_0\circ\mu}\oplus^{\Phi\omega_1}W}.
\end{align*}
Explicitly, for $(x,w)\in\gg\oplus W$,
\begin{align*}
(x,w)^{(h,v)}& =(x^h,\rho_0^1(h)^{-1}(w+\dot{\rho}_1(x)v)+\frac{d}{d\tau}\rest{\tau=0}\alpha(h;\exp_G(\tau x))),
\end{align*}
where we also abuse notation and write $x^h$ for the induced action of $H$ 
on $\gg$. By definition, the structural action of the Lie $2$-algebra of 
(\ref{CocGpExt}) is the differential at the identity of 
the group homomorphism
\begin{align*}
\xymatrix{H{}_{\rho^0_0}\ltimes^{\omega_0}V \ar[r] & Aut(\gg{}_{\dot{\rho}^1_0\circ\mu}\oplus^{\Phi\omega_1}W):(h,v) \ar@{|->}[r] & (-)^{(h,v)^{-1}}}.
\end{align*}
Thus, the action by derivations of $\hh_{\dot{\rho}^0_0}\oplus^{\Phi\omega_0}V$ 
on $\gg{}_{\dot{\rho}^1_0\circ\mu}\oplus^{\Phi\omega_1}W$ is given by
\begin{align*}
\Lie_{(y,v)}(x,w)=(\Lie_y x,\dot{\rho}_0^1(y)w-\dot{\rho}_1(x)v+\frac{d}{d\lambda}\rest{\lambda=0}\frac{d}{d\tau}\rest{\tau=0}\alpha(\exp_H(\lambda y)^{-1};\exp_G(\tau x)))
\end{align*}
for $(y,v)\in\hh\oplus V$ and $(x,w)\in\gg\oplus W$. Together with this 
action, (\ref{CocAlgExt}) is naturally seen as an extension of $\gg_1$ by 
$\xymatrix{W \ar[r]^\phi & V}$ whose classifying $2$-cocycle under  
Theorem~\ref{H2Alg} is 
$(\Phi\varphi,\Phi\omega_0,\Phi\alpha,\Phi\omega_1)\in C_\nabla^2(\gg_1,\phi)$.

In the remainder of this section, we prove that $\Phi$ commutes with all  
differentials and difference maps defining $\nabla$, thus defining a map 
of complexes. To illustrate the van Est strategy, we consider two cases 
separately: First, we deal with the case where the representation takes 
values on an honest vector space; then, we move on to the general case.

\subsection{Vector space coefficients}
One way of regarding the category of vector spaces as a subcategory of 
$2$-vector spaces is realizing a vector space $V$ as the groupoid that has 
$V$ as its space of objects and single arrow attached to each object, the 
$2$-term complex of which is $\xymatrix{(0) \ar[r] & V}$. Having $W=(0)$ 
has the effect of collapsing (\ref{Alg3dimLat}) and (\ref{Gp3dimLat}) to 
the page $r=0$. Though in general $\partial$ and $\delta$ do not commute, 
they do so up to isomorphism in the $2$-vector space when $r=0$. Since two 
elements in a vector space are isomorphic if and only if they are the same, 
$\partial$ and $\delta$ commute and form double complexes. We prove that, 
in this case, $\Phi$ yields a map of double complexes. In the generic cube,
\begin{align}\label{r=0Cube}
\xymatrix@!0{
& & C(\G_{p}^{q+1},V) \ar[rrrrr]^\Phi\ar[ddll]_\partial & & & & & \bigwedge^{q+1}\gg_{p}^*\otimes V \ar[ddll]_\partial  \\
 \\
C(\G_{p+1}^{q+1},V) \ar[rrrrr]^\Phi & & & & & \bigwedge^{q+1}\gg _{p+1}^*\otimes V  \\
& & C(\G_{p}^{q},V) \ar'[rrr]^\Phi[rrrrr]\ar'[u][uuu]^\delta\ar[ddll]_\partial & & & & & \bigwedge^{q}\gg_{p}^*\otimes V \ar[uuu]_\delta\ar[ddll]^\partial \\
\\
C(\G_{p+1}^{q},V) \ar[rrrrr]^\Phi  \ar[uuu]^\delta & & & & & \bigwedge^{q}\gg_{p+1}^*\otimes V  \ar[uuu]^\delta
}
\end{align}
the left and right squares commute because they lie in their corresponding 
double complexes. Due to the observation that (\ref{vanEstMap}) coincides 
with the classic Van Est map, the front and back squares commute as well. We 
are left to prove is the following lemma.
\begin{lemma}
In (\ref{r=0Cube}), $\Phi\circ\partial=\partial\circ\Phi$.
\end{lemma} 
\begin{proof}
Each face map $\partial_k$ in a Lie $2$-group is a homomorphism whose 
derivative is the face map $\hat{\partial}_k$ in its Lie $2$-algebra; 
therefore, if $\xi\in\gg_{p+1}$, 
\begin{align}\label{expFaces}
\partial_k(\exp_{\G_{p+1}}(\xi))=\exp_{\G_{p}}(\hat{\partial}_k\xi).
\end{align}
Let $\omega\in C(\G_p^q,V)$ and $\Xi=(\xi_1,...,\xi_q)\in\gg_{p+1}^q$. 
Then
\begin{align*}
\overrightarrow{R}_\Xi(\partial\omega) & =\frac{d}{d\lambda_q}\rest{\lambda_q=0}...\frac{d}{d\lambda_1}\rest{\lambda_1=0}\sum_{k=0}^{p+1}(-1)^k\omega(\partial_k\exp_{\G_{p+1}}(\lambda_1\xi_{1}),...,\partial_k\exp_{\G_{p+1}}(\lambda_q\xi_{q)})) \\
									   & =\sum_{k=0}^{p+1}(-1)^k\frac{d}{d\lambda_q}\rest{\lambda_q=0}...\frac{d}{d\lambda_1}\rest{\lambda_1=0}\omega(\exp_{\G_{p}}(\lambda_1\hat{\partial}_k\xi_1),...,\exp_{\G_{p}}(\lambda_q\hat{\partial}_k\xi_q)) 
									     =\sum_{k=0}^{p+1}(-1)^k\overrightarrow{R}_{\partial_k\Xi}\omega ,
\end{align*}
and
\begin{align*}
\Phi(\partial\omega)(\Xi) & =\sum_{\sigma\in S_q}\abs{\sigma}\overrightarrow{R}_{\sigma(\Xi)}(\partial\omega) 
						   =\sum_{\sigma\in S_q}\abs{\sigma}\sum_{k=0}^{p+1}(-1)^k\overrightarrow{R}_{\partial_k\sigma(\Xi)}\omega 
						   =\sum_{k=0}^{p+1}(-1)^k(\Phi\omega)(\partial_k\Xi)=\partial(\Phi\omega)(\Xi).
\end{align*}
\end{proof}

In this case, we have the following van Est type theorem.

\begin{theorem}\label{2vE-vs} 
Let $\G$ be a Lie $2$-group with associated crossed module 
$\xymatrix{G \ar[r] & H}$, Lie $2$-algebra $\gg_1$ and a representation on 
the $2$-vector space $\xymatrix{(0) \ar[r] & V}$. If $H$ is $k$-connected, 
$G$ is $(k-1)$-connected and $\Phi$ is the van Est map (\ref{vanEstMap}), 
then the total cohomology of its mapping cone double vanishes: 
\begin{eqnarray*}
H^n_{tot}(\Phi)=(0),\quad\textnormal{for all degrees } n\leq k.
\end{eqnarray*}
\end{theorem}
\begin{proof}
The first page of the spectral sequence of the filtration by columns of 
$C^{\bullet,\bullet}(\Phi)$ is (\ref{1stPage}). Recall that 
$\G_p\cong G^p\times H$; hence, using the K\"unneth formula and the 
connectedness hypotheses for $H$ and $G$, it follows that $\G_p$ is 
$(k-1)$-connected. Theorem~\ref{vanEst} implies all columns in $E^{p,q}_1$ 
vanish below $k-1$, and so the result follows from Lemma~\ref{BelowDiag}.

\end{proof}

Rephrasing with Proposition~\ref{ConeCoh}:

\begin{corollary}\label{vanEstVs}
Under the hypotheses of Theorem~\ref{2vE-vs}, the van Est map 
(\ref{vanEstMap}) induces isomorphisms
\begin{eqnarray*}
\xymatrix{
\Phi :H_\nabla^n(\G ,V) \ar[r] & H_\nabla^n(\gg_1 ,V),
}
\end{eqnarray*}
for $n\leq k$, and it is injective for $n=k+1$.
\end{corollary}

As an application, we prove the following partial integrability result.

\begin{theorem}
Let $\gg_1$ be a Lie $2$-algebra with associated crossed module 
$\xymatrix{\gg \ar[r]^\mu & \hh}$. If 
\begin{align}\label{hyp}
\gg\cap\mathfrak{c}(\hat{u}(\hh)) & =(0),
\end{align}
where $\mathfrak{c}(\hat{u}(\hh))$ is the centralizer of $\hat{u}(\hh)$ in 
$\gg_1$, then $\gg_1$ is integrable. 
\end{theorem}
\begin{proof}
We are to use the van Est strategy. Consider the exact sequence 
\begin{eqnarray}\label{vE2-Ext}
\xymatrix{
0 \ar[r] & \ker(\ad_1) \ar[d]\ar@{^{(}->}[r] & \gg \ar[d]_\mu\ar[r]^{\ad_1\quad} & \ad_1(\gg) \ar[d]\ar[r] & 0  \\
0 \ar[r] & \ker(\ad_0) \ar@{^{(}->}[r]       & \hh \ar[r]_{\ad_0\quad}           & \ad_0(\hh) \ar[r]       & 0, 
}
\end{eqnarray}
associated to the adjoint representation of Example~\ref{2ad}. If 
$x\in\gg$, from Eq.~(\ref{hyp}), there exists a $y_x\in\hh$ such that 
$[x,\hat{u}(y_x)]_1=\ad_1(x)(y_x)\neq 0$. Consequently, 
$\ad_1(x)\neq 0\in Hom(\hh,\gg)$ for every $x\in\gg$, and $\ker(\ad_1)=(0)$. 
Let $[\omega]\in H_{\nabla}^2(\gg_1,\ker(\ad_0))$ be the class corresponding 
to (\ref{vE2-Ext}) under Theorem~\ref{H2Alg}. 
 
The image of any linear functor between $2$-vector spaces yields a Lie 
subgroupoid, and Lie $2$-subalgebras of $\ggl(\phi)$ can be integrated using 
exponentials \cite{Sheng_Zhu2:2012}; hence, 
$\xymatrix{\ad_1(\gg) \ar[r] & \ad_0(\hh)}$ is integrable to a Lie $2$-group 
$\G$ with associated crossed module $\xymatrix{G \ar[r] & H}$. Picking $G$ 
and $H$ $1$-connected, we may use Corollary~\ref{vanEstVs} to conclude  
\begin{align*}
[\omega] & =\Phi[\smallint\omega], & \textnormal{for a unique } & [\smallint\omega]\in H_{\nabla}^2(\G ,\ker(\ad_0)).
\end{align*}
The extension of $\G$ by $\ker(\ad_0)$ that corresponds to 
$[\smallint\omega]$ under Theorem~\ref{H2Gp} integrates $\gg_1$.

\end{proof}
\begin{remark}
One could cast Eq.~(\ref{hyp}) in terms of the isotropy Lie algebras of the 
action of $\hh$ on $\gg$ by asking equivalently that 
$\dim\lbrace y\in\hh :\Lie_y x=0\rbrace>0$ for all $x\in\gg$.
\end{remark}

\subsection{The general case}
We devote the remainder of this section to prove that, when values are taken 
on a $2$-vector space $\xymatrix{W \ar[r]^\phi & V}$ with $W\neq (0)$, the 
van Est map $\Phi$ (\ref{vanEstMap}) still defines a map of complexes. In 
particular, we show that $\Phi$ commutes with all differentials in 
Subsection~\ref{sss-Cxs} and all difference maps necessary to define the 
complexes (\ref{Cxs}) up to degree $2$. The proofs we present boil down to 
long and unfortunately unenlightening computations. 

For notational convenience, we adopt the following shorthands. For an index 
set $I=\lbrace 1,...,n\rbrace$, 
\begin{align*}
\frac{d^I}{d\tau_I}\rest{\tau_I=0}:=\frac{d}{d\tau_n}\rest{\tau_n=0}\cdots\frac{d}{d\tau_1}\rest{\tau_1=0}.
\end{align*}
On the other hand, for any Lie algebra $\gg$, if $X=(x_1,...,x_n)\in\gg^n$, 
we define
\begin{align*}
\exp(\tau_I\cdot X):=(\exp_G(\tau_1x_1),...,\exp_G(\tau_nx_n))\in G^n,
\end{align*}
where $G$ is a Lie group integrating $\gg$.

We also use the following partitions of the symmetric group $S_n$. 
First,
\begin{align*}
S_n=\bigcup_{a=1}^nS_{n-1}(m\vert a),
\end{align*} 			
where $S_{n-1}(m\vert a)$ is the set of permutations that fix the $m$th 
element to be $a$; in symbols, 
$S_{n-1}(m\vert a):=\lbrace\sigma\in S_n:\sigma(m)=a\rbrace$.
Each element $\sigma\in S_{n-1}(m\vert a)$ can be factored as 
$\sigma=\sigma'\circ\sigma^m_a$ where,
\begin{align*}
\sigma^m_a(j):=\begin{cases} 
a & \textnormal{if }j=m \\
j-1 & \textnormal{if }m<j\leq a \\
j+1 & \textnormal{if }a\leq j<m  \\
j & \textnormal{otherwise.}\end{cases}
\end{align*}
The residual permutation $\sigma'$ leaves the $m$th element alone and shifts 
the remaining $n-1$ elements; thus, one can regard $\sigma'$ as belonging to 
the permutation group $S_{n-1}$, incidentally justifying the notation. Since 
$\sigma^m_a$ is a composition of $\abs{m-a}$ transpositions, 
$\abs{\sigma}=\abs{\sigma'}\abs{\sigma^m_a}=(-1)^{m-a}\abs{\sigma'}$.

Iterating this process, for any $k<n$ and $a_0<...<a_{k-1}$, one can 
partition the symmetric group as
\begin{align*}
S_n=\bigcup_{a_0<...<a_{k-1}}\bigcup_{\varrho\in S_k}S_{n-k}(m\vert a_{\varrho(0)}...a_{\varrho(k-1)}),
\end{align*} 
where 
$S_{n-k}(m\vert a_0...a_{k-1}):=\lbrace\sigma\in S_n:\sigma(m+j)=a_j,\forall 0\leq j<k\rbrace$.	
Each element $\sigma\in S_{n-k}(m\vert a_{\varrho(0)}...a_{\varrho(k-1)})$ 
can be factored as $\sigma=\sigma^{(r)}\circ\sigma^m_{a_0...a_{k-1}}\circ\varrho$, 
where $\varrho$ is interpreted to act only on 
$\lbrace a_0,...,a_{k-1}\rbrace$ and 
$\sigma^m_{a_0...a_{k-1}}=\sigma^{m+k-1}_{a_{k-1}}\circ\cdots\circ\sigma^{m+1}_{a_1}\circ\sigma^m_{a_0}$.
Again, the residual permutation $\sigma^{(r)}$ leaves fixed the $k$ elements 
following $m$ and shifts the remaining $n-k$ elements, thus allowing it to 
be regarded as belonging to $S_{n-k}$. The sign is computed to be
\begin{align*}
\abs{\sigma}=\abs{\sigma^{(r)}}\abs{\sigma^m_{a_0...a_{k-1}}}\abs{\varrho}=(-1)^{km+\frac{(k-1)k}{2}-(a_0+...+a_{k-1})}\abs{\sigma^{(r)}}\abs{\varrho}.
\end{align*}

{\bf The r-direction} - The following results prove that the van Est map 
$\Phi$ (\ref{vanEstMap}) commutes with the differentials in the 
$r$-direction. 

\begin{lemma}\label{delta'}
Let $\omega\in C^{p,q}_0(\G,\phi)$, then
\begin{align*}
\Phi(\delta'\omega)=\delta'(\Phi\omega)\in C^{p,q}_1(\gg_1,\phi).
\end{align*}
\end{lemma}
\begin{proof}
For $q=0$, $\omega=v\in V$ and $z\in\gg$,
\begin{align*}
\Phi(\delta'v)(z) & =\frac{d}{d\tau}\rest{\tau=0}(\delta'v)(\exp(\tau z))=\frac{d}{d\tau}\rest{\tau=0}\rho_1(\exp(\tau z))v=\dot{\rho}_1(z)v=(\delta'v)(z),
\end{align*}
and $\Phi$ is defined to be the identity when $(q,r)=(0,0)$. 

$t_p$ is a composition of face maps and so is its derivative $\hat{t}_p$; 
therefore, it is a group homomorphism and, if $\xi\in\gg_p$,
\begin{align}\label{expFinTarget}
t_p(\exp_{\G_p}(\xi))=\exp_{H}(\hat{t}_p\xi).
\end{align}
For $q>0$, let $\Xi=(\xi_1,...,\xi_q)\in\gg_p^q$, then setting 
$h_j:=\prod_{k=j}^q\exp(\lambda_k\hat{t}_p(\xi_k))=\exp(\lambda_{j}\hat{t}_p(\xi_{j}))...\exp(\lambda_q\hat{t}_p(\xi_q))\in H$, we compute
\begin{align*}
\overrightarrow{R}_{\Xi}R_z(\delta'\omega) & =\frac{d^I}{d\lambda_I}\rest{\lambda_I=0}\frac{d}{d\tau}\rest{\tau=0}(\delta'\omega)(\exp(\lambda_I\cdot\Xi);\exp(\tau z))=\frac{d^I}{d\lambda_I}\rest{\lambda_I=0}\rho_0^1(h_1)^{-1}\dot{\rho}_1(z)\omega(\exp(\lambda_I\cdot\Xi))	\\
										   & =\rho_0^1(h_2)^{-1}\dot{\rho}_1(z)R_{\xi_1}\omega(\exp(\lambda_{2}\xi_{2}),...,\exp(\lambda_q\xi_q));
\end{align*}
inductively implying 
$\overrightarrow{R}_{\Xi}R_z(\delta'\omega)=\dot{\rho}_1(z)(\overrightarrow{R}_{\Xi}\omega)$.
Thus,
\begin{align*}
\Phi(\delta'\omega)(\Xi;z) & =\sum_{\sigma\in S_q}\abs{\sigma}\overrightarrow{R}_{\sigma(\Xi)}R_z(\delta'\omega)=\sum_{\sigma\in S_q}\abs{\sigma}\dot{\rho}_1(z)(\overrightarrow{R}_{\sigma(\Xi)}\omega)=\dot{\rho}_1(z)\big{(}(\Phi\omega)(\Xi)\big{)}=\delta'(\Phi\omega)(\Xi;z).
\end{align*}

\end{proof}

\begin{lemma}\label{delk}
Let $G$ be a Lie group with Lie algebra $\gg$ and let $V$ be a vector space.  
If $\omega\in C(G^p,V)$,  $X=(x_0,...,x_p)\in\gg^{p+1}$, 
and $\partial_k$ is the $k$th face map (\ref{faceMaps}) with $0<k<p+1$, 
then for all $m<n$, 
\begin{align*}
\sum_{\sigma\in S_{p-1}(k-1\vert mn)\cup S_{p-1}(k-1\vert nm)}\abs{\sigma}\overrightarrow{R}_{\sigma(X)}\partial_k^*\omega & =(-1)^{m+n}\sum_{\varrho\in S_{p-1}}\abs{\varrho}\overrightarrow{R}_{\varrho(X^{(k)}_{mn})}\omega,
\end{align*}
where, we interpret $S_{p-1}$ as the space of bijections between the set of 
coordinates of $X(k-1,k)$ and the set of coordinates of $X(m,n)$, and  
$\varrho(X^{(k)}_{mn}):=(\varrho(x_0),...,\varrho(x_{k-2}),[x_m,x_n],\varrho(x_{k+1}),...,\varrho(x_p))$. 
\end{lemma}
\begin{proof}
Let $\sigma\in S_{p-1}(k-1\vert mn)$, then 
$\sigma=\sigma''\circ\sigma^{k-1}_{mn}$ and there exists a unique 
$\bar{\sigma}\in S_{p-1}(k-1\vert nm)$ given by 
$\bar{\sigma}=\sigma''\circ\sigma^{k-1}_{nm}$.
By definition, if $\varphi\in C(G,V)$,  
\begin{align*}
R_{[x_m,x_n]}\varphi & =\frac{d}{d\tau_2}\rest{\tau_2=0}\frac{d}{d\tau_1}\rest{\tau_1=0}\varphi(\exp(\tau_1x_n)\exp(\tau_2x_m))-\varphi(\exp(\tau_1x_m)\exp(\tau_2x_n));
\end{align*} 
hence, regarding $\sigma''$ as belonging to $S_{p-1}$,
\begin{align*}
\abs{\sigma}\overrightarrow{R}_{\sigma(X)}\partial_k^*\omega +\abs{\bar{\sigma}}\overrightarrow{R}_{\bar{\sigma}(X)}\partial_k^*\omega & =(-1)^{m+n}\abs{\sigma''}\overrightarrow{R}_{\sigma''(X^{k}_{mn})}\omega.
\end{align*}

\end{proof}

\begin{proposition}\label{delta1}
Let $r>0$ and $\omega\in C^{p,q}_r(\G,\phi)$, then
\begin{align*}
\Phi(\delta_{(1)}\omega)=\delta_{(1)}(\Phi\omega)\in C^{p,q}_{r+1}(\gg_1,\phi).
\end{align*}
\end{proposition}
\begin{proof}
Let $\Xi=(\xi_1,...,\xi_q)\in\gg_p^q$ and $Z=(z_0,...,z_r)\in\gg^{r+1}$. 
$\overrightarrow{R}_\Xi\overrightarrow{R}_Z(\delta_{(1)}\omega)$ has got 
three types of terms corresponding respectively to the dual face maps 
$\delta_0^*$, $\delta_k^*$ for $1\leq k\leq r$ and $\delta_{r+1}^*$ in 
$\delta_{(1)}\omega$: 
\begin{align*}
I   & :=\frac{d^I}{d\lambda_I}\rest{\lambda_I=0}\frac{d^J}{d\tau_J}\rest{\tau_J=0}\rho_0^1(i(\exp(\tau_0z_0)^{\prod_{j=1}^q t_p(\exp(\lambda_j\xi_j))}))\omega(\exp(\lambda_I\cdot\Xi);\exp(\tau_1z_1),...,\exp(\tau_rz_r)) \\
II  & :=\frac{d^I}{d\lambda_I}\rest{\lambda_I=0}\frac{d^J}{d\tau_J}\rest{\tau_J=0}\omega(\exp(\lambda_I\cdot\Xi);\exp(\tau_0z_0),...,\exp(\tau_kz_k)\exp(\tau_{k+1}z_{k+1}),...,\exp(\tau_rz_r)) \\
III & :=\frac{d^I}{d\lambda_I}\rest{\lambda_I=0}\frac{d^J}{d\tau_J}\rest{\tau_J=0}\omega(\exp(\lambda_I\cdot\Xi);\exp(\tau_0z_0),...,\exp(\tau_{r-1}z_{r-1})).
\end{align*}
Type $III$ terms are constant with respect to $\tau_r$ and thus vanish. 

For any fixed $0\leq k\leq r$, partition 
$S_{r+1}=\bigcup_{m<n}(S_{r-1}(k-1\vert mn)\cup S_{r-1}(k-1\vert nm))$ and 
use Lemma~\ref{delk} to conclude 
\begin{align*}
\sum_{\varrho\in S_{r+1}}\abs{\varrho}\overrightarrow{R}_{\varrho(Z)}\delta_k^*\omega & =\sum_{m<n}(-1)^{m+n}\sum_{\varrho'\in S_{r-1}}\abs{\varrho'}\overrightarrow{R}_{\varrho'(Z^{(k)}_{mn})}\omega.
\end{align*}
For each pair $m<n$, using $S_r=\bigcup_{k=1}^rS_{r-1}(0\vert k)$, 
\begin{align*}
\sum_{\varrho\in S_r}\abs{\varrho}\overrightarrow{R}_{\varrho([z_m,z_n],Z(m,n))}\omega=\sum_{k=1}^r\sum_{\varrho'\in S_{r-1}}(-1)^k\abs{\varrho'}\overrightarrow{R}_{\varrho'(Z^{(k)}_{mn})}\omega;
\end{align*}
thus, summing all type $II$ terms yields
\begin{align*}
\sum_{\sigma\in S_q}\sum_{\varrho\in S_{r+1}}\sum_{k=1}^r(-1)^k\abs{\sigma}\abs{\varrho}\overrightarrow{R}_{\sigma(\Xi)}\overrightarrow{R}_{\varrho(Z)}(\delta_k^*\omega)=\sum_{m<n}(-1)^{m+n}(\Phi\omega)(\Xi;[z_m,z_n],Z(m,n)).
\end{align*}

Concludingly, recall that for $y\in\hh$ and $x\in\gg$, 
\begin{align*}
\frac{d}{d\lambda}\rest{\lambda=0}\frac{d}{d\tau}\rest{\tau=0}\rho_0^1(i(\exp(\tau x)^{\exp(\lambda y)})) & =\frac{d}{d\lambda}\rest{\lambda=0}\dot{\rho}_0^1(\mu(x^{\exp(\lambda y)}))=-\dot{\rho}_0^1(\mu(\Lie_yx));
\end{align*}
hence, setting  
$h_j:=\prod_{k=j}^q\exp(\lambda_{k}\hat{t}_p(\xi_{k}))=\exp(\lambda_{j}\hat{t}_p(\xi_{j}))...\exp(\lambda_q\hat{t}_p(\xi_q))\in H$,
we compute 
\begin{align*}
\frac{d}{d\lambda_1}\rest{\lambda_1=0}\dot{\rho}_0^1(\mu(z_0^{h_1}))(\overrightarrow{R}_{Z(0)}\omega)(\exp(\lambda_I\cdot\Xi)) 
	& =\dot{\rho}_0^1(\mu(z_0^{\exp(0)h_2}))R_{\xi_1}(\overrightarrow{R}_{Z(0)}\omega)(\exp(\lambda_{2}\xi_{2}),...,\exp(\lambda_q\xi_q))+ \\
	& \qquad -\dot{\rho}_0^1(\mu((\Lie_{\hat{t}_p(\xi_1)}z_0)^{h_2}))(\overrightarrow{R}_{Z(0)}\omega)(\exp(0),\exp(\lambda_{2}\xi_{2}),...,\exp(\lambda_q\xi_q)) \\
  & =\dot{\rho}_0^1(\mu(z_0^{h_2}))R_{\xi_1}(\overrightarrow{R}_{Z(0)}\omega)(\exp(\lambda_{2}\xi_{2}),...,\exp(\lambda_q\xi_q))
\end{align*}
and inductively, 
$I=\dot{\rho}_0^1(\mu(z_0))\overrightarrow{R}_{\Xi}\overrightarrow{R}_{Z(0)}\omega$.
Using the partition of $S_{r+1}$ by $S_r(0\vert k)$'s and summing type 
$I$ terms yields
\begin{align*}
\sum_{\sigma\in S_{q}}\sum_{\varrho\in S_{r+1}}\abs{\sigma}\abs{\varrho}\overrightarrow{R}_{\sigma(\Xi)}\overrightarrow{R}_{\varrho(Z)}(\delta_0^*\omega) 
						 & =\sum_{\sigma\in S_{q}}\sum_{k=0}^r\sum_{\varrho'\in S_r(0\vert k)}(-1)^{k}\abs{\sigma}\abs{\varrho'}\dot{\rho}_0^1(\mu(z_k))\overrightarrow{R}_{\sigma(\Xi)}\overrightarrow{R}_{\varrho'(Z(k))}\omega \\
						 & =\sum_{k=0}^r(-1)^k\dot{\rho}_0^1(\mu(z_{k}))(\Phi\omega)(\Xi;Z(k)),
\end{align*}
and the result follows.
			 
\end{proof}

{\bf The q-direction} - The following results prove that the van Est map 
$\Phi$ (\ref{vanEstMap}) commutes with the differentials in the 
$q$-direction. Since the differentials in the $r$-direction do commute 
with differentials in the $q$-direction, we prove that for constant $p$, 
$\Phi$ yields a map of double complexes that we refer to as $p$-pages.

\begin{lemma}\label{Multilin}
Let $\xymatrix{T:V\times ...\times V \ar[r] & W}$ be an $r$-multilinear map. 
If $H_\lambda$ is a differentiable path of automorphisms of $V$ with 
$H_0=Id_V$, then
\begin{align*}
\frac{d}{d\lambda}\rest{\lambda=0}T(H_\lambda(v_1),...,H_\lambda(v_r)) & =\sum_{k=1}^rT(v_1,...,v_{k-1},\frac{d}{d\lambda}\rest{\lambda=0}H_{\lambda}(v_k),v_{k+1},...,v_r)
\end{align*}
\end{lemma}
\begin{proof}
Let $\lbrace e_i\rbrace_{i=1}^n$ be a basis for $V$. In these coordinates, 
$R_{a_1...a_r}:=R(e_{a_1},...,e_{a_r})$ and 
$H_\lambda(e_a)=H_a^b(\lambda)e_b$. Since $H_0=Id_V$, $H_a^b(0)=\delta_a^b$, 
on basic elements,
\begin{align*}
\frac{d}{d\lambda}\rest{\lambda=0}R(H_\lambda(e_{a_1}),...,H_\lambda(e_{a_r})) & =\frac{d}{d\lambda}\rest{\lambda=0}R(H_{a_1}^{b_1}(\lambda)e_{b_1},...,H_{a_r}^{b_r}(\lambda)e_{b_r})=\frac{d}{d\lambda}\rest{\lambda=0}H_{a_1}^{b_1}(\lambda)...H_{a_r}^{b_r}(\lambda)R_{b_1...b_r} \\
 & =R_{b_1...b_r}\sum_{k=1}^rH_{a_1}^{b_1}(0)...H_{a_{k-1}}^{b_{k-1}}(0)\dot{H}_{a_k}^{b_k}(0)H_{a_{k+1}}^{b_{k+1}}(0)...H_{a_r}^{b_r}(0) \\
 & =\sum_{k=1}^rR_{a_1...a_{k-1}b_ka_{k+1}...a_r}\dot{H}_{a_k}^{b_k}(0) \\
 & =\sum_{k=1}^rR(e_{a_1},...,e_{a_{k-1}},\frac{d}{d\lambda}\rest{\lambda=0}H_{a_k}^{b_k}(\lambda)e_{b_k},e_{a_{k+1}},...,e_{a_r}), 
\end{align*}
as desired.
\end{proof}		

\begin{proposition}\label{delta}
Let $r>0$ and $\omega\in C^{p,q}_r(\G,\phi)$, then
\begin{align*}
\Phi(\delta\omega)=\delta(\Phi\omega)\in C^{p,q+1}_r(\gg_1,\phi).
\end{align*}
\end{proposition}
\begin{proof}
Let $\Xi=(\xi_0,...,\xi_q)\in\gg_p^{q+1}$ and $Z=(z_1,...,z_r)\in\gg^r$. 
$\overrightarrow{R}_\Xi\overrightarrow{R}_Z(\delta\omega)$ has got 
three types of terms corresponding respectively to the dual face maps 
$\delta_0^*$, $\delta_j^*$ for $1\leq j\leq q$ and $\delta_{q+1}^*$ in 
$\delta\omega$: 
\begin{align*}
I   & :=\frac{d^I}{d\lambda_I}\rest{\lambda_I=0}\frac{d^J}{d\tau_J}\rest{\tau_J=0}\omega(\exp(\lambda_1\xi_1),...,\exp(\lambda_{q}\xi_{q});\exp(\tau_J\cdot Z)^{t_p(\exp(\lambda_0\xi_0))}) \\
II  & :=\frac{d^I}{d\lambda_I}\rest{\lambda_I=0}\frac{d^J}{d\tau_J}\rest{\tau_J=0}\omega(\exp(\lambda_0\xi_0),...,\exp(\lambda_{j-1}\xi_{j-1})\exp(\lambda_j\xi_j),...,\exp(\lambda_q\xi_q);\exp(\tau_J\cdot Z)) \\
III & :=\frac{d^I}{d\lambda_I}\rest{\lambda_I=0}\frac{d^J}{d\tau_J}\rest{\tau_J=0}\rho_0^1(t_p(\exp(\lambda_q\xi_q)))^{-1}\omega(\exp(\lambda_0\xi_0),...,\exp(\lambda_{q-1}\xi_{q-1});\exp(\tau_J\cdot Z)).
\end{align*}
Here, we used the notation $(g_1,...,g_r)^h:=(g_1^h,...,g_r^h)$ for 
$g_1,...,g_r\in G$ and $h\in H$. 

Using Eq.~(\ref{expFinTarget}), 
$III=-\dot{\rho}_0^1(\hat{t}_p(\xi_q))\overrightarrow{R}_{\Xi(q)}\overrightarrow{R}_Z\omega$; 
hence, using $S_{q+1}=\bigcup_{j=0}^qS_q(q\vert j)$ and summing type $III$ 
terms yields
\begin{align}\label{III}
\sum_{\sigma\in S_{q+1}} & \sum_{\varrho\in S_r}\abs{\sigma}\abs{\varrho}\frac{d}{d\lambda_{q+1}}\rest{\lambda_{q+1}=0}\rho_0^1(t_p(\exp(\lambda_q\xi_{\sigma(q)})))^{-1}\overrightarrow{R}_{\sigma(\Xi)}\overrightarrow{R}_{\varrho(Z)}(\delta_{q+1}^*\omega) \\
						 & =\sum_{j=0}^q\sum_{\sigma'\in S_{q}(q\vert j)}\sum_{\varrho\in S_r}(-1)^{q-j+1}\abs{\sigma'}\abs{\varrho}\dot{\rho}_0^1(\hat{t}_p(\xi_{j}))\overrightarrow{R}_{\sigma'(\Xi(j))}\overrightarrow{R}_{\varrho(Z)}\omega 
						   =\sum_{j=0}^q(-1)^{q-j+1}\dot{\rho}_0^1(\hat{t}_p(\xi_{j}))(\Phi\omega)(\Xi(j);Z).\nonumber
\end{align}

Let $\vec{\gamma}\in\G_p^q$, then 
$(\overrightarrow{R}_\bullet\omega)(\vec{\gamma})\in\bigwedge^r\gg^*\otimes W$; 
in particular, it is $r$-multilinear. Due to Eq.~(\ref{expFinTarget}), 
consider the differentiable path $(-)^{\exp(\lambda_0\hat{t}_p(\xi_0))}$  
of automorphisms through the identity and invoke Lemma~\ref{Multilin} to 
conclude 
\begin{align*}
\frac{d^I}{d\lambda_I}\rest{\lambda_I=0}\overrightarrow{R}_{Z^{\exp(\lambda_0\hat{t}_p(\xi_0))}}(\delta_0^*\omega)(\exp(\lambda_I\cdot\Xi)) & =-\sum_{k=1}^r\overrightarrow{R}_{\Xi(0)}R_{z_r}...R_{z_{k+1}}R_{\Lie_{\hat{t}_p(\xi_0)}z_k}R_{z_{k-1}}...R_{z_1}\omega. 
\end{align*}
Using $S_{q+1}=\bigcup_{j=0}^qS_q(0\vert j)$, the sum type $I$ terms 
yields
\begin{align}\label{I}
\sum_{\sigma\in S_{q+1}} & \sum_{\varrho\in S_r}\abs{\sigma}\abs{\varrho}\frac{d}{d\lambda_0}\rest{\lambda_0=0} \overrightarrow{R}_{\sigma(\Xi)}\overrightarrow{R}_{\varrho(Z)^{\exp(\lambda_0\hat{t}_p(\xi_{\sigma(0)}))}}(\delta_0^*\omega) \\
                         & =\sum_{j=0}^q\sum_{\sigma'\in S_{q}(0\vert j)}\sum_{\varrho\in S_r}\sum_{k=1}^r(-1)^{j+1}\abs{\sigma'}\abs{\varrho}\overrightarrow{R}_{\sigma'(\Xi(0))}\overrightarrow{R}_{z_{\varrho(r)}}...R_{z_{\varrho(k+1)}}R_{\Lie_{\hat{t}_p(\xi_j)}z_{\varrho(k)}}R_{z_{\varrho(k-1)}}...R_{z_{\varrho(1)}}\omega\nonumber \\
& =\sum_{j=0}^q\sum_{k=1}^r(-1)^{j+1}(\Phi\omega)(\Xi(j);z_1,...,\Lie_{\hat{t}_p(\xi_j)}z_k,...,z_r).\nonumber
\end{align}
Adding together (\ref{I}) and $(-1)^{q+1}$(\ref{III}), we get 
$\sum_{j=0}^q(-1)^{j+1}\rho^{(r)}(\xi_j)(\Phi\omega)(\Xi(j);Z)$ (cf. 
Eq.~(\ref{q-rep})).

Using a reasoning parallel to that in the proof of Proposition~\ref{delta1} 
for type $II$ terms, one concludes 
\begin{align*}
\sum_{\sigma\in S_{q+1}}\sum_{\varrho\in S_r}\sum_{j=1}^q(-1)^j\abs{\sigma}\abs{\varrho}\overrightarrow{R}_{\sigma(\Xi)}\overrightarrow{R}_{\varrho(Z)}(\delta_j^*\omega)=\sum_{m<n}(-1)^{m+n}(\Phi\omega)([\xi_m,\xi_n],\Xi(m,n);Z),
\end{align*}
and the result follows.
						 
\end{proof}

\begin{theorem}\label{p-pag}
For constant $p$, $\Phi$ restricts to a map $\Phi_p$ (\ref{Phi_p}) of 
double complexes.
\end{theorem}

{\bf The p-direction} - The following results prove that the van Est map 
$\Phi$ (\ref{vanEstMap}) commutes with the differentials in the 
$p$-direction. Since the differentials in the $r$-direction do commute 
with differentials in the $p$-direction, we prove that for constant $q$, 
$\Phi$ yields a map of double complexes that we refer to as $q$-pages.

\begin{proposition}\label{partial}
Let $r>0$ and $\omega\in C^{p,q}_r(\G,\phi)$, then
\begin{align*}
\Phi(\partial\omega)=\partial(\Phi\omega)\in C^{p+1,q}_r(\gg_1,\phi).
\end{align*}
\end{proposition}
\begin{proof}
Let $\Xi=(\xi_1,...,\xi_q)\in\gg_{p+1}^q$ and $Z\in\gg^r$. Writting  
$\exp_{\G_{p+1}}(\xi)=(e^\xi_1,...,e^\xi_{p+1})\in\G_{p+1}\leq\G^{p+1}$, 
for $\xi\in\gg_{p+1}$, define
\begin{align*}
g_j & :=pr_G\Big{(}e^{\lambda_j\xi_j}_1\vJoin e^{\lambda_{j+1}\xi_{j+1}}_1\vJoin\cdots\vJoin e^{\lambda_{q-1}\xi_{q-1}}_1\vJoin e^{\lambda_q\xi_q}_1\Big{)}\in G, & \textnormal{for } & 1\leq j\leq q.   
\end{align*}
Using Eq.~(\ref{expFaces}),
\begin{align*}
\overrightarrow{R}_{\Xi}\overrightarrow{R}_Z(\partial\omega) & =\frac{d^I}{d\lambda_I}\rest{\lambda_I=0}\Big{[}\rho_0^1(i(g_1))^{-1}(\overrightarrow{R}_Z\omega)(\partial_0\exp(\lambda_I\cdot\Xi))+\sum_{j=1}^{p+1}(-1)^j(\overrightarrow{R}_Z\omega)(\partial_j\exp(\lambda_I\cdot\Xi))\Big{]} \\
	& =\sum_{j=1}^{p+1}(-1)^j\overrightarrow{R}_{\hat{\partial}_j\Xi}\overrightarrow{R}_Z\omega+\frac{d^I}{d\lambda_I}\rest{\lambda_I=0}\rho_0^1(i(g_1))^{-1}(\overrightarrow{R}_Z\omega)(\exp(\lambda_I\cdot\hat{\partial}_0\Xi)).
\end{align*}
Now,
\begin{align*}
\frac{d}{d\lambda_1}\rest{\lambda_1=0}\rho_0^1(i(g_1))^{-1}(\overrightarrow{R}_Z\omega)(\exp(\lambda_I\cdot & \hat{\partial}_0\Xi))=\rho_0^1(i(g_2))^{-1}R_{\hat{\partial}_0\xi_1}(\overrightarrow{R}_Z\omega)\big{(}\exp(\lambda_2\hat{\partial}_0\xi_2),...,\exp(\lambda_q\hat{\partial}_0\xi_q)\big{)}+ \\
	& +\Big{(}\frac{d}{d\lambda_1}\rest{\lambda_1=0}\rho_0^1(i(g_1))^{-1}\Big{)}(\overrightarrow{R}_Z\omega)\big{(}\exp(0),\exp(\lambda_2\hat{\partial}_0\xi_2),...,\exp(\lambda_q\hat{\partial}_0\xi_q)\big{)};
\end{align*}
thus, inductively,
$\overrightarrow{R}_\Xi\overrightarrow{R}_Z(\partial\omega)=\sum_{j=0}^{p+1}(-1)^j\overrightarrow{R}_{\hat{\partial}_j\Xi}\overrightarrow{R}_Z\omega$, and as a consequence,
\begin{align*}
\Phi(\partial\omega)(\Xi;Z) & =\sum_{\sigma\in S_q}\sum_{\varrho\in S_r}\sum_{j=0}^{p+1}(-1)^j\abs{\sigma}\abs{\varrho}\overrightarrow{R}_{\sigma(\hat{\partial}_j\Xi)}\overrightarrow{R}_{\varrho(Z)}\omega=\sum_{j=0}^{p+1}(-1)^j\Phi\omega(\hat{\partial}_j\Xi;Z)=\partial(\Phi\omega)(\Xi;Z).
\end{align*}

\end{proof}

\begin{theorem}\label{q-pag}
For constant $q$, $\Phi$ restricts to a map of double complexes.
\end{theorem}

{\bf Difference maps} - We conclude this Section by proving that the van Est 
map $\Phi$ (\ref{vanEstMap}) commutes with the difference maps. We restrict 
our attention to the difference maps necessary to prove that $\Phi$ defines 
a map of complexes up to degree $2$.

Under the isomorphisms of Subsection~\ref{sss-Equiv}, 
$\xi=(\xi_1,...,\xi_p)\in\gg_p$ corresponds to a unique element 
$(x_1,...,x_p;y)\in\gg^p\oplus\hh$, where each individual $\xi_m\in\gg_1$ 
corresponds to $(x_m,y_m)=(x_m,y+\mu(x_{m+1}+...+x_p))\in\gg\oplus\hh$. In 
the upcoming computations, we often need to consider $\xi$ as the sum
\begin{align}\label{BreakXi}
\xi & =(X_1^r,0)+(0,\hat{\partial}_0^r\xi)=(x_1,...,x_r,0,...,0;0)+(0,...,0,x_{r+1},...,x_p;y),
\end{align}
where $X_1^r:=(x_1,...,x_r)\in G^r$. Here, we abused notation and wrote an 
$=$ sign to mean the image under the isomorphism of Subsection 
\ref{sss-Equiv}. Since the inclusion of $G$ in $\G$ is a group homomorphism, 
and the full nerve of the Lie $2$-group lies in the category of Lie groups, 
\begin{align}\label{degMaps}
\exp_{\G_p}(0,\hat{\partial}_0^r\xi)& =(1,\exp_{\G_{p-r}}(\hat{\partial}_0^r\xi))=(1,\partial_0^r\exp_{\G_{p-r}}(\xi)) \\ \exp_{\G_p}(0,...,0,x,0,...,0;0) & =(1,...,1,\exp_G(x),1,...,1;1)\in G^p\times H\cong\G_p. \nonumber
\end{align} 
In the latter equation, if $x\in\gg$ is in the $k$th position, so is 
$\exp_G(x)\in G$. Note that, for $1<r\leq p$, none of the inclusions of 
$G^r$ in $\G_p$ is a Lie group homomorphism; hence, there is no relation 
analog to (\ref{degMaps}) when there is more than one non-zero entry.

\begin{proposition}\label{Delta}
Let $r>0$ and $\omega\in C^{p,q}_r(\G,\phi)$, then
\begin{align*}
\Phi(\Delta\omega)=\Delta(\Phi\omega)\in C^{p+1,q+1}_{r-1}(\gg_1,\phi).
\end{align*}
\end{proposition}
\begin{proof}
If $r=1$, let $\Xi=(\xi_0,...,\xi_q)^T\in\gg_{p+1}^{q+1}$. If 
$\xi_j\in\gg_{p+1}$ corresponds to 
$(x_j^0,...,x_j^p;y_j)\in\gg^{p+1}\oplus\hh$, using the convention of 
Eq.~(\ref{BreakXi}), write $\Xi=\Xi_1+\Xi_2$, where 
\begin{align}\label{brkAgain}
\Xi_1 & =\begin{pmatrix}
x_0^0\quad 0 \\
 \Xi(0)
\end{pmatrix} & \Xi_2 & =\begin{pmatrix}
0\quad\hat{\partial}_0\xi_0 \\
 \Xi(0) 
\end{pmatrix}.
\end{align}
Since $\overrightarrow{R}_\bullet(\Delta\omega)\in\bigotimes^{q+1}\gg_{p+1}^*\otimes V$, 
$\overrightarrow{R}_\Xi(\Delta\omega)=\overrightarrow{R}_{\Xi_1}(\Delta\omega)+\overrightarrow{R}_{\Xi_2}(\Delta\omega)$. 
Now, from Eq.~(\ref{degMaps}), $c_{11}(\Xi_2)=1$; hence, 
$\overrightarrow{R}_{\Xi_2}(\Delta\omega)=0$. On the other hand, 
letting $I'=\lbrace 1,...,q\rbrace$, 
\begin{align*}
\overrightarrow{R}_{\Xi_1}(\Delta\omega) & =\frac{d^I}{d\lambda_I}\rest{\lambda_I=0}\rho_0^0(t_p(\partial_0\exp(\lambda_1\xi_1))...t_p(\partial_0\exp(\lambda_q\xi_q)))\circ\phi\big{(}\omega(\partial_0\exp(\lambda_{I'}\cdot\Xi(0));\exp_G(\lambda_0x_0^0))\big{)} \\
	& =\frac{d^{I'}}{d\lambda_{I'}}\rest{\lambda_{I'}=0}\rho_0^0(t_p(\partial_0\exp(\lambda_1\xi_1))...t_p(\partial_0\exp(\lambda_q\xi_q)))\circ\phi\big{(}(R_{x_0^0}\omega)(\partial_0\exp(\lambda_{I'}\cdot\Xi(0)))\big{)}.
\end{align*}
Computing, for any $h\in H$ and letting $I''=\lbrace 2,...,q\rbrace$,
\begin{align*}
\frac{d}{d\lambda_1}\rest{\lambda_1=0}\rho_0^0(t_p(\partial_0\exp(\lambda_1\xi_1)) & h)\circ\phi\big{(}(R_{x_0^0}\omega)(\partial_0\exp(\lambda_{I'}\cdot\Xi(0)))\big{)}=\rho_0^0(h)\circ\phi\big{(}(R_{\hat{\partial}_0\xi_1}R_{x_0^0}\omega)(\partial_0\exp(\lambda_{I''}\cdot\Xi(0,1)))\big{)}+ \\
	& +\Big{(}\frac{d}{d\lambda_1}\rest{\lambda_1=0}\rho_0^0(t_p(\partial_0\exp(\lambda_1\xi_1))h)\Big{)}\circ\phi\big{(}(R_{x_0^0}\omega)(\partial_0\exp(0),\partial_0\exp(\lambda_{I''}\cdot\Xi(0,1)))\big{)};
\end{align*}
inductively implying 
$\overrightarrow{R}_{\Xi_1}(\Delta\omega)=\phi\big{(}\overrightarrow{R}_{\hat{\partial}_0\Xi(0)}R_{x_0^0}\omega\big{)}$. 
Using the partition $S_{q+1}=\bigcup_{j=0}^qS_q(0\vert j)$, one 
concludes
\begin{align*}
\Phi(\Delta\omega)(\Xi) & =\sum_{j=0}^q\sum_{\sigma'\in S_q(0\vert j)}(-1)^j\abs{\sigma'}\phi\big{(}\overrightarrow{R}_{\sigma'(\hat{\partial}_0\Xi(j))}R_{x_j^0}\omega\big{)}=\sum_{j=0}^q(-1)^j(\Phi\omega)(\hat{\partial}_0\Xi(j);x_j^0)=\Delta(\Phi\omega)(\Xi).
\end{align*}

If $r>1$, let $Z=(z_1,...,z_{r-1})\in\gg^{r-1}$ and $\Xi$ as before. This 
time round,
$\overrightarrow{R}_\bullet\overrightarrow{R}_Z(\Delta\omega)\in\bigotimes^{q+1}\gg_{p+1}^*\otimes W$;
thus, using Eq.~(\ref{brkAgain}), 
$\overrightarrow{R}_\Xi\overrightarrow{R}_Z(\Delta\omega)=\overrightarrow{R}_{\Xi_1}\overrightarrow{R}_Z(\Delta\omega)+\overrightarrow{R}_{\Xi_2}\overrightarrow{R}_Z(\Delta\omega)$.
The fact that $c_{11}(\Xi_2)=1$ remains; hence, 
$\overrightarrow{R}_{\Xi_2}(\Delta\omega)=0$. Now, the terms of type 
$c_{2n}$ are constant with respect to $\tau_{r-1}$ and thus vanish. Writing 
$\gamma(\lambda_0)$ for $(\exp_G(\lambda_0x_0^0)\quad 1)$, we compute 
\begin{align*}
\frac{d}{d\lambda_0} & \rest{\lambda_0=0}\frac{d^J}{d\tau_J}\rest{\tau_J=0}c_{2n-1}(\exp(\tau_J\cdot Z);\gamma(\lambda_0)) \\
	& =\frac{d}{d\lambda_0}\rest{\lambda_0=0}\frac{d^J}{d\tau_J}\rest{\tau_J=0}(\exp(\tau_J\cdot Z)_{[1,r-n)}^{i(\exp(\lambda_0x_0^0))},\exp(\lambda_0x_0^0)^{-1},\exp(\tau_J\cdot Z)_{[r-n,r-2]},\exp(\tau_{r-1}z_{r-1})\exp(\lambda_0x_0^0)) \\
	& =-\frac{d}{d\lambda_0}\rest{\lambda_0=0}\frac{d^J}{d\tau_J}\rest{\tau_J=0}(\exp(\tau_J\cdot Z)_{[1,r-n)}^{i(\exp(0))},\exp(\lambda_0x_0^0),\exp(\tau_J\cdot Z)_{[r-n,r-2]},\exp(\tau_{r-1}z_{r-1})\exp(0)).
\end{align*}
Ultimately, putting 
$\varrho'(Z)^{(j)}_{x_j^0}:=(z_{\varrho'(1)},...,z_{\varrho'(j-1)},x_j^0,z_{\varrho'(j)},...,z_{\varrho'(r-1)})$,
and using successively the partitions 
$S_{q+1}=\bigcup_{j=0}^qS_q(0\vert j)$ and $S_{r}=\bigcup_{n=0}^{r-1}S_{r-1}(j\vert n)$, 
we compute
\begin{align*}
\Phi(\Delta\omega)(\Xi;Z) & =\sum_{j=0}^q\sum_{\sigma'\in S_q(0\vert j)}\sum_{\varrho'\in S_{r-1}}(-1)^j\abs{\sigma'}\abs{\varrho'}\overrightarrow{R}_{\sigma'(\Xi(j))}R_{\xi_j}\overrightarrow{R}_{\varrho'(Z)}(\Delta\omega) \\
	& =\sum_{j=0}^q\sum_{\sigma'\in S_q(0\vert j)}\sum_{\varrho'\in S_{r-1}}\sum_{n=0}^{r-1}(-1)^{j+n}\abs{\sigma'}\abs{\varrho'}\overrightarrow{R}_{\sigma'(\partial_0\Xi(j))}\overrightarrow{R}_{\varrho'(Z)^{(j)}_{x_j^0}}\omega \\
	& =\sum_{j=0}^q\sum_{\sigma'\in S_q(0\vert j)}\sum_{\varrho\in S_r}(-1)^j\abs{\sigma'}\abs{\varrho}\overrightarrow{R}_{\sigma'(\partial_0\Xi(j))}\overrightarrow{R}_{\varrho(x_j^0,Z)}\omega =\sum_{j=0}^q(-1)^j(\Phi\omega)(\partial_0\Xi(j);x_j^0,Z),
\end{align*}
which yields $\Delta(\Phi\omega)(\Xi;Z)$ as desired.

\end{proof}

\begin{proposition}\label{Deltar1}
Let $r>1$ and $\omega\in C^{p,q}_r(\G,\phi)$, then
\begin{align*}
\Phi(\Delta_{r,1}\omega)\equiv0\in C^{p+r,q+1}_0(\gg_1,\phi).
\end{align*}
\end{proposition}
\begin{proof}
Let $\Xi=(\xi_0,...,\xi_q)^T\in\gg_{p+r}^{q+1}$ and 
$(x_0^1,...,x_0^{p+r};y_0)\in\gg^{p+r}\oplus\hh$ be the image of 
$\xi_0\in\gg_{p+r}$ under the isomorphism of Subsection~\ref{sss-Equiv}. 
Using the convention of Eq.~(\ref{BreakXi}), write $\Xi=\Xi_0+...+\Xi_r$, 
where 
\begin{align*}
\Xi_0 & =\begin{pmatrix}
0\quad\hat{\partial}_0^r\xi_0 \\
\Xi(0)
\end{pmatrix} & \Xi_k & =\begin{pmatrix}
0\cdots0\quad x_0^k\quad 0\cdots0 \\
\Xi(0)
\end{pmatrix},\quad\text{for }1\leq k\leq r.
\end{align*}
Since 
$\overrightarrow{R}_\bullet(\Delta_{r,1}\omega)\in\bigotimes^{q+1}\gg_{p+r}\otimes V$,
$\overrightarrow{R}_\Xi(\Delta_{r,1}\omega)=\overrightarrow{R}_{\Xi_0}(\Delta_{r,1}\omega)+...+\overrightarrow{R}_{\Xi_r}(\Delta_{r,1}\omega)$. 
It then follows from Eq.~(\ref{degMaps}) that
\begin{align*}
\overrightarrow{R}_{\Xi_0}(\Delta_{r,1}\omega)
	& =\frac{d^I}{d\lambda_I}\rest{\lambda_I=0}\rho_0^0(t_p(\partial_0^r\exp(\lambda_0\xi_0))...t_p(\partial_0^r\exp(\lambda_q\xi_q)))\circ\phi\big{(}\omega(\partial_0^r\delta_0\exp(\lambda_I\cdot\Xi_0);1,...,1)\big{)}=0,
\end{align*}
and, for $1\leq k\leq r$,
\begin{align*}
\overrightarrow{R}_{\Xi_k} & (\Delta_{r,1}\omega) \\
	& =\frac{d^I}{d\lambda_I}\rest{\lambda_I=0}\rho_0^0(t_p(\partial_0^r\exp(\lambda_1\xi_1))...t_p(\partial_0^r\exp(\lambda_q\xi_q)))\circ\phi\big{(}\omega(\partial_0^r\delta_0\exp(\lambda_I\cdot\Xi_k);1,...,\exp_G(\lambda_0x_0^k),...,1)\big{)}=0.
\end{align*}

\end{proof}

\begin{proposition}\label{Delta1r}
Let $r>1$ and $\omega\in C^{p,q}_r(\G,\phi)$, then
\begin{align*}
\Phi(\Delta_{1,r}\omega)=(-1)^{\frac{r(r+1)}{2}}\Delta_r(\Phi\omega)\in C^{p+1,q+r}_0(\gg_1,\phi).
\end{align*}
\end{proposition}
\begin{proof}
Let $\Xi=(\xi_1,...,\xi_{q+r})^T\in\gg_{p+1}^{q+r}$ and 
$(x_j^0,...,x_j^p;y_j)\in\gg^{p+1}\oplus\hh$ be the image of 
$\xi_j\in\gg_{p+1}$ under the isomorphism of Subsection~\ref{sss-Equiv}. 
For $\xi\in\gg_{p+1}$, write 
$(g_\xi^0,...,g_\xi^p;h_\xi)\in G^{p+1}\times H$ for the image of 
$\exp_{\G_{p+1}}(\xi)\in\G_{p+1}$ under the isomorphism of Subsection 
\ref{sss-Equiv}.

Making $\Xi=\Xi_1+\Xi_2$ in the manner of Eq.~(\ref{brkAgain}), 
$\overrightarrow{R}_\Xi(\Delta_{1,r}\omega)=\overrightarrow{R}_{\Xi_1}(\Delta_{1,r}\omega)+\overrightarrow{R}_{\Xi_2}(\Delta_{1,r}\omega)$
as  
$\overrightarrow{R}_\bullet(\Delta_{1,r}\omega)\in\bigotimes^{q+r}\gg_{p+1}\otimes V$.
Given that
\begin{align*}
\omega(\partial_0\delta_0^r\exp(\lambda_I\cdot\Xi_2);1^{h_{\lambda_2\xi_2}...h_{\lambda_{q+r}\xi_{q+r}}},(g_{\lambda_2\xi_2}^0)^{h_{\lambda_2\xi_2}...h_{\lambda_{q+r}\xi_{q+r}}},...,(g_{\lambda_{q+r-1}\xi_{q+r-1}}^0)^{h_{\lambda_{q+r}\xi_{q+r}}},g_{\lambda_{q+r}\xi_{q+r}}^0) & =0,
\end{align*}
$\overrightarrow{R}_{\Xi_2}(\Delta_{1,r}\omega)=0$; therefore, inductively 
using Eq.~(\ref{brkAgain}) on $\Xi(1,...,k)$, one ultimately concludes 
$\overrightarrow{R}_{\Xi}(\Delta_{1,r}\omega)=\overrightarrow{R}_{\Xi_{(r)}}(\Delta_{1,r}\omega)$,
where 
\begin{align*}
\Xi_{(r)}:=\begin{pmatrix}
x_1^0\quad 0 \\
\vdots \\
x_r^0\quad 0 \\
\Xi(1,...,r)
\end{pmatrix} ,
\end{align*}
and, as in the proof of Proposition~\ref{Delta}, 
\begin{align*}
\overrightarrow{R}_{\Xi_{(r)}}(\Delta_{1,r}\omega) & =\phi\big{(}\overrightarrow{R}_{\hat{\partial}_0\Xi(1,...,r)}\overrightarrow{R}_{X_{1,...,r}^0}\omega\big{)}, & \text{where } & X_{1,...,r}^0=(x_1^0,...,x_r^0)\in\gg^r.
\end{align*}
Partitioning 
$S_{q+r}=\bigcup_{a_1<...<a_r}\bigcup_{\varrho\in S_r}S_q(1\vert a_{\varrho(1)}...a_{\varrho(r)})$,
\begin{align*}
\Phi(\Delta_{1,r}\omega)(\Xi) & =\sum_{a_1<...<a_r}\sum_{\varrho\in S_r}\sum_{\sigma^{(r)}\in S_q(1\vert a_{\varrho(1)}...a_{\varrho(r)})}(-1)^{a_1+...+a_r+\frac{r(r+1)}{2}}\abs{\sigma^{(r)}}\abs{\varrho}\phi\big{(}\overrightarrow{R}_{\sigma^{(r)}(\hat{\partial}_0\Xi(1,...,r))}\overrightarrow{R}_{X_{a_{\varrho(1)},...,a_{\varrho(r)}}^0}\omega\big{)} \\
	& =\sum_{a_1<...<a_r}(-1)^{a_1+...+a_r+\frac{r(r+1)}{2}}\phi\big{(}(\Phi\omega)(\partial_0(\Xi(a_1,...,a_k));X_{a_1,...,a_r}^0)\big{)}=(-1)^{\frac{r(r+1)}{2}}\Delta_r(\Phi\omega)(\Xi).
\end{align*}
\end{proof}
All the previous results add up to the following theorem.
\begin{theorem}\label{itIs}
The van Est map $\Phi$ (\ref{vanEstMap}) induces a map of complexes 
\begin{align*}
\Phi & :\xymatrix{C_\nabla^n(\G,\phi) \ar[r] & C_\nabla^n(\gg_1,\phi)}
\end{align*}
between the complexes (\ref{Cxs}) for $n\leq 2$.
\end{theorem}
\begin{remark}
Continuing Remark~\ref{Incomplete}, Theorem~\ref{itIs} extends to $n\leq 5$ 
for the difference maps made explicit in \cite{Angulo2:2020}. We abstain 
from presenting a proof because it lies outside of our application.
\end{remark}

\section{A Collection of van Est Type Theorems}\label{sec-Theos}
In this section, we prove a van Est type theorem relating the cohomology 
of Lie $2$-groups and Lie $2$-algebras:
\begin{theorem}\label{2-vanEstTheo}
Let $\G$ be a Lie $2$-group with associated crossed module 
$\xymatrix{G \ar[r] & H}$, Lie $2$-algebra $\gg_1$ and a representation on 
the $2$-vector space $\xymatrix{W \ar[r]^\phi & V}$. If $H$ and $G$ are both 
$k$-connected and the van Est map $\Phi$ (\ref{vanEstMap}) induces a map of 
complexes between the complexes (\ref{Cxs}) for $n\leq k+1$, then
\begin{align*}
\xymatrix{
\Phi:H_\nabla^n(\G ,\phi) \ar[r] & H_\nabla^n(\gg_1 ,\phi),
}
\end{align*}
is an isomorphism for $n\leq k$ and it is injective for $n=k+1$.
\end{theorem}
Theorem~\ref{2-vanEstTheo} follows from the vanishing of the cohomology of 
the mapping cone of $\Phi$, which in turn, using the spectral sequence of 
the filtration of (\ref{Cllpsed}) by columns, follows from van Est type 
theorems that ensure the vanishing of its columns below the diagonal. 
Throughout, fix a Lie $2$-group $\G$ with Lie $2$-algebra $\gg_1$ and 
a representation $\rho$ on $\xymatrix{W \ar[r]^\phi & V}$.

\subsection{First approximation}\label{subsec-1stAprox}
Momentarily disregarding the full crossed module structure, consider only 
the Lie group $H$ acting to the right by Lie group automorphisms of $G$. In 
\cite{Angulo2:2020}, it is explained that $H\ltimes G$ has got the structure 
of a double Lie groupoid in the sense of Ehresmann \cite{Ehresmann:1963}, 
where the top groupoid is a Lie group bundle over $H$ and the left groupoid 
is the right action groupoid of $H$ over $G$. Furthermore, for $r>0$, the 
$p$-pages of $C^{p,q}_r(\G,\phi)$ can be thought of as naturally associated 
double complexes induced by a map of double Lie groupoids 
\begin{align}\label{DblRepn}
\xymatrix{
 & H\ltimes G \ar@<0.5ex>[dl]\ar@<-0.5ex>[dl]\ar@<0.5ex>[dd]\ar@<-0.5ex>[dd]\ar[r] & GL(W)\ltimes GL(W) \ar@<0.5ex>[dr]\ar@<-0.5ex>[dr]\ar@<0.5ex>[dd]\ar@<-0.5ex>[dd] & \\
H \ar[rrr]^{\rho_H\qquad\quad}\ar@<0.5ex>[dd]\ar@<-0.5ex>[dd] & & & GL(W) \ar@<0.5ex>[dd]\ar@<-0.5ex>[dd]   \\
 &             G \ar@<0.5ex>[dl]\ar@<-0.5ex>[dl]\ar[r]^{\rho_G\quad} & GL(W) \ar@<0.5ex>[dr]\ar@<-0.5ex>[dr] & \\       
 \ast \ar[rrr]                                                     & & & \ast ,                  
}	
\end{align} 
where $W$ is a vector space and the double Lie groupoid to the right is 
the one given by the right action by conjugation of $GL(W)$ on itself (cf. 
(\ref{p-pagRepn})).
 
Associated to a double Lie groupoid, there are two LA-groupoids 
\cite{Mackenzie:1992}, which are roughly given by passing the Lie functor in 
the vertical and the horizontal directions. In the case of $H\ltimes G$, 
these are respectively given by 
\begin{align}\label{LAGpds}
& \xymatrix{
\hh\ltimes G \ar@<0.5ex>[r]\ar@<-0.5ex>[r]\ar[d] & \hh \ar[d] \\       
G \ar@<0.5ex>[r]\ar@<-0.5ex>[r]                  & \ast   } & & \xymatrix{
H\ltimes \gg \ar@<0.5ex>[r]\ar@<-0.5ex>[r]\ar[d] & \gg \ar[d] \\       
H \ar@<0.5ex>[r]\ar@<-0.5ex>[r]                  & \ast .
}
\end{align}
Bear in mind that the notation $\ltimes$ stands alternatively for the 
transformation groupoid associated to an action of a Lie group and the 
action Lie algebroid associated to the action of a Lie algebra. 
Correspondingly, there are two morphisms of LA-groupoids which correspond to 
the differentiation of the map (\ref{DblRepn}) in each direction. The 
subsequent results say there are double complexes naturally associated to 
each of the latter derivatives. 
\begin{proposition}\label{Ver p-page}
Let 
$\xymatrix{\hh\ltimes G \ar[r] & \ggl(W)\ltimes GL(W):(y;g) \ar@{|->}[r] & (\rho_\hh(y),\rho_G(g))}$ 
be the differentiation of the map (\ref{DblRepn}) in the vertical direction. 
Then, for every $q\geq 0$, there is a representation $\rho_G^q$ of the Lie 
group bundle 
\begin{align}\label{verGpBdlRep}
 & \xymatrix{\hh^q\times G \ar@<0.5ex>[r] \ar@<-0.5ex>[r] & \hh^q} & \text{on } & \xymatrix{\hh^q\times W \ar[r] & \hh^q},
\end{align} 
and for every $r\geq 0$, there is a representation $\rho^r_\hh$ of the 
action Lie algebroid 
\begin{align}\label{verActAlgbdRep}
 & \xymatrix{\hh\ltimes G^r \ar[r] & G^r} & \text{on } & \xymatrix{G^r\times W \ar[r] & G^r}
\end{align}
such that the grid 
\begin{align}\label{protoVerDbl}
\xymatrix{ \vdots & \vdots & \vdots &  \\ 
\bigwedge^3\hh^*\otimes W \ar[r]^{\partial\quad}\ar[u] & C(G,\bigwedge^3\hh^*\otimes W) \ar[r]^{\partial}\ar[u] & C(G^2,\bigwedge^3\hh^*\otimes W) \ar[r]\ar[u] & \dots \\
\bigwedge^2\hh^*\otimes W \ar[r]^{\partial\quad}\ar[u]^{\delta} & C(G,\bigwedge^2\hh^*\otimes W) \ar[r]^{\partial}\ar[u]^{\delta} & C(G^2,\bigwedge^2\hh^*\otimes W) \ar[r]\ar[u]^{\delta} & \dots \\
\hh^*\otimes W  \ar[r]^{\partial\quad}\ar[u]^{\delta} & C(G,\hh^*\otimes W) \ar[r]^{\partial}\ar[u]^{\delta} & C(G^2,\hh^*\otimes W) \ar[r]\ar[u]^{\delta} & \dots \\
W \ar[r]^{\partial\quad}\ar[u]^{\delta} & C(G,W) \ar[r]^{\partial}\ar[u]^{\delta} & C(G^2,W) \ar[r]\ar[u]^{\delta} & \dots }
\end{align}
whose rows are the subcomplexes of alternating $q$-multilinear Lie groupoid 
cochains with values in $\rho_G^q$ and whose columns are Lie algebroid 
complexes with values on $\rho_\hh^r$ is a double complex.
\end{proposition}
\begin{proof}
Since (\ref{DblRepn}) is a map of double groupoids, $\rho_\hh\times\rho_G$ 
is a map of LA-groupoids and its restrictions to the base Lie group and to 
the side Lie algebra give respectively representations $\rho_G$ of $G$ and 
$\rho_\hh$ of $\hh$, both on $W$. These are the representations for the 
group bundle over a point 
$\xymatrix{G \ar@<0.5ex>[r] \ar@<-0.5ex>[r] & \ast}$ and for the trivial 
action Lie algebroid $\xymatrix{\hh\ltimes\ast \ar[r] & \ast}$. For each 
$q>0$, define $\rho_G^q :=pr_G^*\rho_G$, which is a representation of 
the group bundle (\ref{verGpBdlRep}) as the projection onto $G$ is a Lie 
groupoid homomorphism. Analogously, for each $r>0$, define 
$\rho_\hh^r:=\hat{t}_r^*\rho_\hh$, where 
$\hat{t}_r:(\hh\times G)_r\cong\xymatrix{\hh\times G^r \ar[r] & \hh}$ is 
the final target map of the LA-groupoid and hence a Lie algebroid map.

Since for any group bundle, $(\hh^q\times G)_r\cong\hh^q\times G^r$, 
$C^r(\hh^q\times G;\hh^q\times W)\cong C(\hh^q\times G^r,W)$, one can make 
sense of the subspace of alternating $q$-multilinear in the 
$\hh$-coordinates $r$-cochains of (\ref{verGpBdlRep}) with values on 
$\rho_G^q$,  
\begin{align*}
C^r_{lin}(\hh^q\times G,W) & :=\lbrace\omega\in C^r(\hh^q\times G;\hh^q\times W):\omega(-;\vec{g})\in\bigwedge^q\hh^*\otimes W,\forall\vec{g}\in G^r\rbrace .
\end{align*}
$C^\bullet_{lin}(\hh^q\times G,W)$ is a subcomplex of 
$C^\bullet(\hh^q\times G;\hh^q\times W)$. Indeed, for $Y\in\hh^q$ and 
$\vec{g}=(g_0,...,g_r)\in G^{r+1}$, 
$\partial_k(Y;\vec{g})=(Y;\delta_k\vec{g})$, where $\delta_k$ is the $k$th
face map in the nerve of the Lie group $G$. Thus, for 
$\omega\in C^r_{lin}(\hh^q\times G,W)$,
\begin{align*}
\partial\omega(Y;\vec{g}) & =\rho_G^q(Y;g_0)\omega(Y;\delta_0\vec{g})+\sum_{k=1}^{r+1}(-1)^k\omega(Y;\delta_k\vec{g}),
\end{align*}
but $\rho_G^q(Y;g_0)=\rho_G(g_0)$; hence, $\partial\omega(-;\vec{g})$ is a 
linear combination of alternating $q$-multilinear maps. Since by definition 
\begin{align*}
C^r_{lin}(\hh^q\times G,W)=C(G^r,\bigwedge^q\hh^*\otimes W)=\Gamma\Big{(}\bigwedge^q(\hh\ltimes G^r)^*\otimes(G^r\times W)\Big{)},
\end{align*}
for fixed $(q,r)$, the spaces of $q$-multilinear $r$-cochains of the Lie 
groupoid (\ref{verGpBdlRep}) with values on $\rho_G^q$ and of $q$-cochains 
of the Lie algebroid (\ref{verActAlgbdRep}) with values on $\rho_\hh^r$ 
coincide. 

We are left to prove that the generic square
\begin{align}\label{generic Ver p-square}
\xymatrix{ C(G^r,\bigwedge^{q+1}\hh^*\otimes W) \ar[r]^{\partial}            & C(G^{r+1},\bigwedge^{q+1}\hh^*\otimes W)    \\
C(G^r,\bigwedge^q\hh^*\otimes W) \ar[r]_{\partial}\ar[u]^{\delta} & C(G^{r+1},\bigwedge^q\hh^*\otimes W) \ar[u]_{\delta}  }
\end{align}
commutes. Given that the brackets of all the Lie algebroids involved are 
completely determined by the bracket of $\hh$, we restrict to prove the 
commutativity of (\ref{generic Ver p-square}) for constant sections. Let 
$\omega\in C(G^r,\bigwedge^q\hh^*\otimes W)$, $Y=(y_0,...,y_q)\in\hh^{q+1}$ 
and $\vec{g}$ as above. Then, 
\begin{align*}
\partial(\delta\omega)(Y;\vec{g}) 
              & =\rho_G^{q+1}(Y;g_0)\Big{[}\sum_{j=0}^{q}(-1)^{j+1}\rho_\hh^{r}(y_j)\omega(Y(j);\delta_0\vec{g})+\sum_{m<n}(-1)^{m+n}\omega([y_m,y_n],Y(m,n);\delta_0\vec{g})\Big{]}+ \\
			  & \qquad +\sum_{k=1}^{r+1}(-1)^{k}\Big{[}\sum_{j=0}^{q}(-1)^{j+1}\rho_\hh^{r}(y_j)\omega(Y(j);\delta_k\vec{g})+\sum_{m<n}(-1)^{m+n}\omega([y_m,y_n],Y(m,n);\delta_k\vec{g})\Big{]},										   
\end{align*}
and
\begin{align*}
\delta(\partial & \omega)(Y;\vec{g})=\sum_{j=0}^{q}(-1)^{j+1}\rho_\hh^{r+1}(y_j)\Big{[}\rho_G^{q}(Y(j);g_0)\omega(Y(j);\delta_0\vec{g})+ \sum_{k=1}^{r+1}(-1)^{k}\omega(Y(j);\delta_k\vec{g})\Big{]}+ \\
              & +\sum_{m<n}(-1)^{m+n}\Big{[}\rho_G^{q}([y_m,y_n],Y(m,n);g_0)\omega([y_m,y_n],Y(m,n);\delta_0\vec{g})+\sum_{k=1}^{r+1}(-1)^{k}\omega([y_m,y_n],Y(m,n);\delta_k\vec{g})\Big{]}.
\end{align*}
The equality follows from noticing the following identities: First, 
$\rho_G^{q+1}(Y;g_0)=\rho_G^q([y_m,y_n],Y(m,n);g_0)=\rho_G(g_0)$. Next, for 
$(y;\vec{g})\in\hh\ltimes G^r$, $\hat{t}_r(y;\vec{g})=y$ and its anchor is 
by definition 
$\frac{d}{d\lambda}\rest{\lambda=0}(\vec{g})^{\exp(\lambda y)}\in T_{\vec{g}}G^r$.
Letting $\lbrace e_a\rbrace$ be a basis for $W$ and 
\begin{align*}
\omega(Y(j);\delta_k\vec{g})=\omega^a(Y(j);\delta_k\vec{g})e_a ,
\end{align*}
using the fact that for all $h\in H$ and all ranging value of $k$, 
$(\delta_k\vec{g})^h=\delta_k(\vec{g})^h$, we conclude
\begin{align*}
\rho_\hh^r(y_j)\omega(Y(j);\delta_k\vec{g})=\omega^a(Y(j);\delta_k\vec{g})\rho_\hh(y_j)e_a+\Big{(}\frac{d}{d\lambda}\rest{\lambda=0}\omega^a\big{(}Y(j);\delta_k(\vec{g})^{\exp(\lambda y_j)}\big{)}\Big{)}e_a=\rho_\hh^{r+1}(y_j)\omega(Y(j),\delta_k\vec{g}).
\end{align*}
Finally, as $\rho_\hh\times\rho_G$ is an LA-groupoid map, the following 
diagram whose vertical maps are the anchors commutes,
\begin{align*}
\xymatrix{
TG \ar[rr]^{d\rho_G\qquad}                              & & TGL(W) \\
\hh\ltimes G \ar[rr]_{\rho_\hh\times\rho_G\qquad}\ar[u] & & \ggl(W)\ltimes GL(W); \ar[u] 
}
\end{align*}
therefore, 
$\frac{d}{d\lambda}\rest{\lambda=0}\rho_G(g^{\exp(\lambda y)})=\rho_G(g)\rho_\hh(y)-\rho_\hh(y)\rho_G(g)$.
Writing $\omega^a=\omega^a(Y(j);\delta_0\vec{g})$, we compute
\begin{align*}
\rho_G^{q+1}(Y;g_0)\rho_\hh^r(y_j)\omega(Y(j);\delta_0\vec{g}) & =\rho_G(g_0)\Bigg{[}\omega^a\rho_\hh(y_j)e_a+\Big{(}\frac{d}{d\lambda}\rest{\lambda=0}\omega^a\big{(}Y(j);(\delta_0\vec{g})^{\exp(\lambda y_j)}\big{)}\Big{)}e_a\Bigg{]};
\end{align*}
while on the other hand, $\rho_G^{q}(Y;g_0)\omega(Y(j);\delta_0\vec{g})=\omega^a\rho_G(g_0)e_a$ and
\begin{align*}
\rho & _\hh^{r+1}(y_j)\rho_G^{q}(Y;g_0)\omega(Y(j);\delta_0\vec{g})=\omega^a\rho_\hh(y_j)\rho_G(g_0)e_a+\Big{(}\frac{d}{d\lambda}\rest{\lambda=0}\omega^a\big{(}Y(j);(\delta_0\vec{g})^{\exp(\lambda y_j)}\big{)}\rho_G\big{(}g_0^{\exp(\lambda y_j)}\big{)}\Big{)}e_a \\
  & =\omega^a\rho_\hh(y_j)\rho_G(g_0)e_a+\Big{(}\frac{d}{d\lambda}\rest{\lambda=0}\omega^a\big{(}Y(j);(\delta_0\vec{g})^{\exp(\lambda y_j)}\big{)}\Big{)}\rho_G(g_0)e_a+\omega^a(Y(j);\delta_0\vec{g})\Big{(}\frac{d}{d\lambda}\rest{\lambda=0}\rho_G\big{(}g_0^{\exp(\lambda y_j)}\big{)}\Big{)}e_a \\
  & =\omega^a\rho_\hh(y_j)\rho_G(g_0)e_a+\rho_G(g_0)\Big{[}\Big{(}\frac{d}{d\lambda}\rest{\lambda=0}\omega^a\big{(}Y(j);(\delta_0\vec{g})^{\exp(\lambda y_j)}\big{)}\Big{)}e_a+\omega^a\rho_\hh(y_j)e_a\Big{]}-\omega^a\rho_\hh(y_j)\rho_G(g_0)e_a,
\end{align*}
ultimately implying the commutativity of the square 
(\ref{generic Ver p-square}). 
\end{proof} 
We aim at approximating the cohomology of the map $\Phi_p$ of Theorem 
\ref{p-pag} assembling van Est maps from $C^{p,\bullet}_\bullet(\G,\phi)$ 
to the complexes of Proposition~\ref{Ver p-page}. Since, each $p$-page 
is induced by the double Lie groupoid map
\begin{align}\label{p-pagRepn}
\xymatrix{
\G_p\ltimes G \ar[r] & GL(W)\ltimes GL(W):(\gamma ;g) \ar@{|->}[r] & (\rho_0^1(t_p(\gamma))^{-1},\rho_0^1(i(g))) ,
}
\end{align} 
replacing the first column of maps by (\ref{Gp1st r}), we introduce a first 
column replacement for the associated double complex of 
Proposition~\ref{Ver p-page}.
\begin{lemma}\label{VerLA r-cx}
Let $\omega\in\bigwedge^q\gg_p^*\otimes V$, $\Xi\in\gg_p^q$ and $g\in G$. 
If $\partial'\omega(\Xi;g):=\rho_1(g)\omega(\Xi)$, then, using the notation 
of Proposition~\ref{Ver p-page}, 
\begin{align*}
\xymatrix{\bigwedge^q\gg_p^*\otimes V \ar[r]^{\partial'\quad} & C_{lin}^1(\gg_p^q\times G,W) \ar[r]^{\partial} & C_{lin}^2(\gg_p^q\times G,W)}
\end{align*}
is a complex.
\end{lemma}
\begin{proof}
First, notice that $\partial'$ is well-defined. Then, letting 
$g_0,g_1\in G$, if follows from Eq.'s~(\ref{GpRho1Homo}) and 
(\ref{GpRepEqns}) that
\begin{align*}
\partial(\partial'\omega)(\Xi;g_0,g_1) & =\rho_G^{q}(\Xi;g_0)\rho_1(g_1)\omega(\Xi)-\rho_1(g_0g_1)\omega(\Xi)+\rho_1(g_0)\omega(\Xi) \\
	& =\rho_0^1(i(g_0))\rho_1(g_1)\omega(\Xi)-(I+\rho_1(g_0)\circ\phi)\rho_1(g_1)\omega(\Xi)=0.
\end{align*} 
\end{proof}
\begin{proposition}\label{(q,r)-VerLAdoubleCx}
For each $p\geq 0$, define $C_{LA}(\gg_p\ltimes G,\phi)$ to be
\begin{align}\label{VerLADbl}
\xymatrix{ \vdots & \vdots & \vdots &  \\ 
\bigwedge^3\gg_p^*\otimes V \ar[r]^{\partial'\quad} \ar[u] & C(G,\bigwedge^3\gg_p^*\otimes W) \ar[r]^{\partial}\ar[u] & C(G^2,\bigwedge^3\gg_p^*\otimes W) \ar[r]\ar[u] & \dots \\
\bigwedge^2\gg_p^*\otimes V \ar[r]^{\partial'\quad}\ar[u]^{\delta} & C(G,\bigwedge^2\gg_p^*\otimes W) \ar[r]^{\partial}\ar[u]^{\delta} & C(G^2,\bigwedge^2\gg_p^*\otimes W) \ar[r]\ar[u]^{\delta} & \dots \\
\gg_p^*\otimes V  \ar[r]^{\partial'\quad}\ar[u]^{\delta} & C(G,\gg_p^*\otimes W) \ar[r]^{\partial}\ar[u]^{\delta} & C(G^2,\gg_p^*\otimes W) \ar[r]\ar[u]^{\delta} & \dots \\
V \ar[r]^{\delta'}\ar[u]^{\delta} & C(G,W) \ar[r]^{\delta_{(1)}}\ar[u]^{\delta} & C(G^2,W)	\ar[r]\ar[u]^{\delta} & \dots , }
\end{align}
where the maps in the first column are given by Lemma~\ref{VerLA r-cx} 
and the rest by Proposition~\ref{Ver p-page}. Then 
$C_{LA}(\gg_p\ltimes G,\phi)$ is a double complex, the vertical LA-double 
complex of $C^{p,\bullet}_\bullet(\G,\phi)$.
\end{proposition}
\begin{proof}
The proof reduces to show that, in the first column, all squares commute. 
For a constant section $\xi\in\Gamma(\gg_p\ltimes G)$ and $g\in G$, one has 
\begin{align*}
\rho_{\gg_p}^1(\xi)\rho_1(g) & =\rho_{\gg_p}(\xi)\rho_1(g)+\frac{d}{d\lambda}\rest{\lambda=0}\rho_1(g^{t_p(\exp(\lambda\xi))})\\
                   & =\dot{\rho}_0^1(\hat{t}_p(\xi))\rho_1(g)+\frac{d}{d\lambda}\rest{\lambda=0}\rho_0^1(t_p(\exp(\lambda\xi)))^{-1}\rho_1(g)\rho_0^0(t_p(\exp(\lambda\xi)) \\
                   & =\dot{\rho}_0^1(\hat{t}_p(\xi))\rho_1(g)-\dot{\rho}_0^1(\hat{t}_p(\xi))\rho_1(g)+\rho_1(g)\dot{\rho}_0^0(\hat{t}_p(\xi))=\rho_1(g)\dot{\rho}_0^0(\hat{t}_p(\xi)).
\end{align*}
If $q=0$, let $v\in V$, then 
$\delta(\delta'v)(\xi;g)=\rho_{\gg_p}^1(\xi)\rho_1(g)v=\rho_1(g)\dot{\rho}_0^0(\hat{t}_p(\xi))v=\partial'(\delta v)(\xi;g)$. 
If $q>0$, let $\omega\in\bigwedge^q\gg_p^*\otimes V$ and 
$\Xi=(\xi_0,...,\xi_q)\in\gg_p^{q+1}$, then
\begin{align*}
\delta(\delta'\omega)(\Xi;g) & =\sum_{j=0}^{q}(-1)^j\rho_{\gg_p}^{1}(\xi_j)\rho_1(g)\omega(\Xi(j))+\sum_{m<n}(-1)^{m+n}\rho_1(g)\omega([\xi_m,\xi_n],\Xi(m,n)) \\
	& =\rho_1(g)\Big{(}\sum_{j=0}^{q}(-1)^{j+1}\dot{\rho}_0^0(\hat{t}_p(\xi_j))\omega(\Xi(j))+\sum_{m<n}(-1)^{m+n}\omega([\xi_m,\xi_n],\Xi(m,n))\Big{)}=\partial'(\delta\omega)(\Xi;g).
\end{align*}
\end{proof}
Let 
\begin{align}\label{VerVanEstMap}
\xymatrix{\Phi_V:C^{p,\bullet}_\bullet(\G,\phi) \ar[r] & C_{LA}(\gg_p\ltimes G,\phi)}
\end{align}
be defined by assembling column-wise van Est maps
\begin{align*}
 & \xymatrix{\Phi_V^0:C(\G_p^q,V) \ar[r] & \bigwedge^q\gg_p^*\otimes V} & \text{and }\qquad & \xymatrix{\Phi_V^r:C(\G_p^\bullet\times G^r,W) \ar[r] & C(G^r,\bigwedge^\bullet\gg_p^*\otimes W).} 
\end{align*}
To describe $\Phi_V^r$ explicitly consider a $q$-cochain 
$\omega\in C(\G_p^q\times G^r,W)$, $q$ sections 
$\xi_1,...,\xi_q\in\Gamma(\gg_p\ltimes G^r)$ and let $\lbrace u_a\rbrace$ 
be a basis for $\gg_p$. Then, writing $(\xi_j)_{\vec{g}}=\xi_j^a(\vec{g})u_a$ 
for $\vec{g}\in G^r$, the right-invariant vector field 
$\vec{\xi}_j\in\mathfrak{X}(\G_p\ltimes G^r)$ associated to $\xi_j$ is 
given by 
\begin{align*}
(\vec{\xi}_j)_{(\gamma;\vec{g})}: & =\frac{d}{d\lambda_a}\rest{\lambda_a=0}\xi_j^a((\vec{g})^{t_p(\gamma)})(\exp_{\G_p}(\lambda_au_a);(\vec{g})^{t_p(\gamma)})\Join(\gamma;\vec{g}) \\
		 & =\xi_j^a((\vec{g})^{t_p(\gamma)})\frac{d}{d\lambda_a}\rest{\lambda_a=0}(\gamma\vJoin\exp_{\G_p}(\lambda_au_a);\vec{g})=\xi_j^a((\vec{g})^{t_p(\gamma)})dL_\gamma(u_a)\in T_\gamma\G_p\leq T_\gamma\G_p\oplus T_{\vec{g}}G^r.
\end{align*}
Consequently, if $\gamma_2,...,\gamma_q\in\G_p$, 
$(R_{\xi_q}\omega)(\gamma_2,...,\gamma_q;\vec{g})=\xi_q^a((\vec{g})^{t_p(\gamma_q)...t_p(\gamma_2)})\frac{d}{d\lambda_a}\rest{\lambda_a=0}\omega(\exp_{\G_p}(\lambda_au_a),\gamma_2,...,\gamma_q;\vec{g})$, 
and
\begin{align*}
R & _{\xi_{q-1}}(R_{\xi_q}\omega)(\gamma_3,...,\gamma_q;\vec{g})=\xi_{q-1}^b((\vec{g})^{t_p(\gamma_q)...t_p(\gamma_3)})\frac{d}{d\lambda_{b}}\rest{\lambda_{b}=0}(R_{\xi_q}\omega)(\exp_{\G_p}(\lambda_{b}u_b),\gamma_3,...,\gamma_q;\vec{g}) \\
 & =\xi_{q-1}^b((\vec{g})^{t_p(\gamma_q)...t_p(\gamma_3)})\frac{d}{d\lambda_{b}}\rest{\lambda_{b}=0}\Big{(}\xi_q^a((\vec{g})^{t_p(\gamma_q)...t_p(\gamma_3)\exp(\lambda_{b}u_b)})\frac{d}{d\lambda_a}\rest{\lambda_a=0}\omega(\exp(\lambda_au_a),\exp(\lambda_bu_b),\gamma_3,...,\gamma_q;\vec{g})\Big{)} \\
 & =(\xi_{q-1}^b\xi_q^a)((\vec{g})^{t_p(\gamma_q)...t_p(\gamma_3)})\frac{d}{d\lambda_{b}}\rest{\lambda_{b}=0}\frac{d}{d\lambda_a}\rest{\lambda_a=0}\omega(\exp_{\G_p}(\lambda_au_a),\exp_{\G_p}(\lambda_bu_b),\gamma_3,...,\gamma_q;\vec{g}).
\end{align*}
Inductively, for $\Xi=(\xi_1,...,\xi_q)$, 
$\overrightarrow{R}_\Xi=(\xi_1^{a_1}...\xi_q^{a_q})(\vec{g})\frac{d^I}{d\lambda_{I}}\rest{\lambda_I=0}\omega(\exp_{\G_p}(\lambda_1u_{a_1}),...,\exp_{\G_p}(\lambda_qu_{a_q});\vec{g})$ and 
\begin{align}\label{PhiVr}
(\Phi_V^r\omega)(\Xi;\vec{g}) & =\sum_{\sigma\in S_q}\abs{\sigma}\frac{d^I}{d\lambda_{I}}\rest{\lambda_I=0}\omega\big{(}\exp(\lambda_I\cdot\sigma(\Xi_{\vec{g}}));\vec{g}\big{)}.
\end{align}

\begin{proposition}
The map $\Phi_V$ (\ref{VerVanEstMap}) is a map of double complexes. 
\end{proposition}
\begin{proof}
Since by definition $\Phi_V$ defines a map of complexes when restricted to 
columns, we prove that it is compatible with the horizontal differentials 
in (\ref{VerLADbl}). 
For $r=0$, let $\omega\in C(\G_p^q,V)$, $\Xi=(\xi_1,...,\xi_q)\in\gg_p^q$ 
and $g\in G$, then
\begin{align*}
\overrightarrow{R}_{\Xi}\delta'\omega(g) & =\frac{d^I}{d\lambda_I}\rest{\lambda=0}\rho_0^1(t_p(\exp_{\G_p}(\lambda_1\xi_1)...\exp_{\G_p}(\lambda_q\xi_q)))^{-1}\rho_1(g)\omega(\exp(\lambda_I\cdot\Xi));
\end{align*}
inductively yielding 
$\overrightarrow{R}_{\Xi}\delta'\omega(g)=\rho_1(g)\overrightarrow{R}_{\Xi}\omega$. 
Taking the alternating sum over $S_q$, 
$\Phi_V^1\delta'\omega =\partial'\Phi_V^0\omega$. 
For $r>0$, let $\omega\in C(\G_p^q\times G^r,W)$, 
$\vec{g}=(g_0,...,g_r)\in G^{r+1}$ and $\Xi$ as above, then
\begin{align*}
\overrightarrow{R}_\Xi(\delta_{(1)}\omega)(\vec{g}) & =\frac{d^I}{d\lambda_{I}}\rest{\lambda_I=0}\rho_0^1\big{(}g_0^{t_p(\exp(\lambda_1\xi_1))...t_p(\exp(\lambda_1\xi_1))}\big{)}\omega(\exp(\lambda_I\cdot\Xi);\delta_0\vec{g})+\sum_{k=1}^{r+1}(-1)^{k}\omega(\exp(\lambda_I\cdot\Xi);\delta_k\vec{g}). 
\end{align*}
Inductively, 
$\frac{d^I}{d\lambda_{I}}\rest{\lambda_I=0}\rho_0^1\big{(}g_0^{t_p(\exp(\lambda_1\xi_1))...t_p(\exp(\lambda_1\xi_1))}\big{)}\omega(\exp(\lambda_I\cdot\Xi);\delta_0\vec{g})=\rho_0^1(i(g_0))\overrightarrow{R}_\Xi\omega(\delta_0\vec{g})$; 
hence, taking the alternating sum over $S_q$ and recalling 
$\rho_G^q(\Xi;g_0)=\rho_0^1(i(g_0))$,
\begin{align*}
\Phi_V^{r+1}(\delta_{(1)}\omega)(\Xi;\vec{g}) & =\rho_G^q(\Xi;g_0)\Phi_V^r\omega(\Xi;\delta_0\vec{g})+\sum_{k=1}^{r+1}(-1)^{k}\Phi_V^r\omega(\Xi;\delta_k\vec{g})=\partial(\Phi_V^r\omega)(\Xi;\vec{g}).
\end{align*}
\end{proof}
\begin{theorem}\label{Ver p-vanEst}
If $\G_p$ is $k$-connected, $H^n_{tot}(\Phi_V)=(0)$ for all degrees 
$n\leq k$.
\end{theorem}
\begin{proof}
We compute the cohomology of the mapping cone double of $\Phi_V$ using the 
spectral sequence of its filtration by columns, whose first page is
\begin{align*}
\xymatrix{ & \vdots & \vdots & \vdots &  \\ 
           & H^2(\Phi_V^0) \ar[r] & H^2(\Phi_V^1) \ar[r] & H^2(\Phi_V^2) \ar[r] & \dots \\
E^{p,q}_1: & H^1(\Phi_V^0) \ar[r] & H^1(\Phi_V^1) \ar[r] & H^1(\Phi_V^2) \ar[r] & \dots \\
           & H^0(\Phi_V^0) \ar[r] & H^0(\Phi_V^1) \ar[r] & H^0(\Phi_V^2) \ar[r] & \dots }
\end{align*}
Since $\Phi_V$ is defined column-wise by the van Est maps $\Phi_V^r$, the 
$r$th column of the mapping cone double coincides with the mapping cone of 
$\Phi_V^r$. Invoking Theorem~\ref{Crainic-vanEstRephrased}, the $r$th 
column of $E^{p,q}_1$ vanishes below $k$; indeed, $\G_p$ is the $s$-fibre 
of $\xymatrix{\G_p\ltimes G^r \ar@<0.5ex>[r]\ar@<-0.5ex>[r] & G^r}$ and is 
$k$-connected by hypothesis. Given that 
$E^{p,q}_1\Rightarrow H^{p+q}_{tot}(\Phi_V)$, the result follows from Lemma 
\ref{BelowDiag}.
\end{proof}

\subsection{Second approximation and the main theorem}\label{subsec-2ndAprox}
Observe that the $q$th row of $C_{LA}(\gg_p\ltimes G,\phi)$ coincides with 
the cochain complex of the Lie group $G$ with values in the representation
\begin{align*}
 & \xymatrix{\rho_{(q)}:G \ar[r] & GL(\bigwedge^q\gg_p^*\otimes W),} & \rho_{(q)}(g)\omega &=\rho_0^1(i(g))\omega\quad\text{for }\omega\in\bigwedge^q\gg_p^*\otimes W
\end{align*}
except in degree $0$. Note that assembling row-wise van Est map extended by the 
identity in degree $0$ defines a map  
\begin{align}\label{RowVanEstMap}
\xymatrix{\Phi_{row}:C_{LA}(\gg_p\ltimes G,\phi) \ar[r] & C^{p,\bullet}_\bullet(\gg_1 ,\phi)}
\end{align}
that casually lands in the $p$-page of the grid of the Lie $2$-algebra. Let
$\xymatrix{\Phi_{row}^q:C(G^r,\bigwedge^q\gg_p^*\times W) \ar[r] & \bigwedge^q\gg_p^*\otimes\bigwedge^r\gg\otimes W}$
be the van Est map defined by the van Est map for $r>0$ and the identity of 
$\bigwedge^q\gg_p^*\times V$ for $r=0$. $\Phi_{row}^q$ is explicitly given by 
\begin{align}\label{PhiRow}
(\Phi_{row}^q\omega)(\Xi;Z) & =\sum_{\varrho\in S_r}\abs{\varrho}\frac{d^J}{d\tau_{J}}\rest{\tau_J=0}\omega\big{(}\Xi;\exp(\tau_J\cdot\varrho(Z_{\vec{\gamma}}))\big{)}
\end{align}
for $\omega\in C(G^r,\bigwedge^q\gg_p^*\times W)$, $\Xi\in\gg_p^q$ 
and $Z=(z_1,...,z_r)\in\gg^r$. 
\begin{proposition}
$\Phi_{row}$ is a map of double complexes.
\end{proposition}
\begin{proof}
Since by definition $\Phi_{row}$ defines a map of complexes when restricted 
to rows and to the first column, we are left to prove that it is compatible 
with the vertical differentials in (\ref{VerLADbl}). Let 
$\omega\in C(G^r,\bigwedge^q\gg_p^*\otimes W)$, 
$\Xi=(\xi_0,...,\xi_q)\in\gg_p^{q+1}$ and $Z=(z_1,...,z_r)\in\gg^r$, then
\begin{align*}
\overrightarrow{R}_Z\delta\omega(\Xi) & =\frac{d^J}{d\tau_J}\rest{\tau_J=0}\sum_{j=0}^{q}(-1)^j\rho_{\gg_p}^r(\xi_j)\omega(\Xi(j);\exp(\tau_J\cdot Z))+\sum_{m<n}\omega([\xi_m,\xi_n],\Xi(m,n);\exp(\tau_J\cdot Z))  .
\end{align*}									
Let $\lbrace e_a\rbrace$ be a basis for $W$ and 
$\omega(\Xi(j);\exp(\tau_J\cdot Z))=\omega^a(\Xi(j);\exp(\tau_J\cdot Z))e_a$.
By definition
\begin{align*}
\rho_{\gg_p}^r(\xi_j)\omega(\Xi(j);\exp(\tau\cdot X))=\omega^a(\Xi(j);\exp(\tau_J\cdot Z))\dot{\rho}_0^1(\hat{t}_p(\xi_j))e_a+\big{(}\frac{d}{d\lambda}\rest{\lambda=0}\omega^a(\Xi(j);\exp(\tau_J\cdot Z)^{t_p(\exp(\lambda\xi_j))})\big{)}e_a ;
\end{align*}
hence, we compute 
\begin{align*}
\frac{d}{d\tau_1}\rest{\tau_1=0}\frac{d}{d\lambda}\rest{\lambda=0}\omega^a( & \Xi(j);\exp(\tau_J\cdot Z)^{t_p(\exp(\lambda\xi_j))}) =\frac{d}{d\lambda}\rest{\lambda=0}R_{x_1^{t_p(\exp(\lambda\xi_j))}}\omega^a(\Xi(j);\exp(\tau_J\cdot Z)(1)^{t_p(\exp(\lambda\xi_j))}) \\
 & =-R_{\Lie_{\hat{t}_p(\xi_j)}x_1}\omega^a(\Xi(j);\exp(\tau_J\cdot Z)(1)))+\frac{d}{d\lambda}\rest{\lambda=0}R_{x_1}\omega^a(\Xi(j);\exp(\tau_J\cdot Z)(1)^{t_p(\exp(\lambda\xi_j))}),
\end{align*}
which inductively yields
\begin{align*}
\frac{d^J}{d\tau_J}\rest{\tau_J=0}\frac{d}{d\lambda}\rest{\lambda=0}\omega^a(\Xi(j);\exp(\tau_J\cdot Z)^{t_p(\exp(\lambda\xi_j))}) & =-\sum_{k=1}^rR_{x_r}...R_{x_{k+1}}R_{\Lie_{\hat{t}_p\xi}x_k}R_{x_{k-1}}...R_{x_1}\omega(\Xi(j)).
\end{align*}
The result now follows from taking the alternating sum over $S_r$.
\end{proof}

That the cohomology of $\Phi_{row}$ vanishes follows from a strategy analog to 
that of Theorem~\ref{Ver p-vanEst}; however, we cannot use the Crainic-van 
Est Theorem directly as we replaced the space of $0$-cochains. In the sequel, 
we prove a van Est type theorem for $\Phi_{row}^q$ that takes into account 
these replacements.
\begin{lemma}\label{tweak0}
Let $\partial'$ be the map of Lemma~\ref{VerLA r-cx} and $\delta'$ the map 
of Eq.~(\ref{Alg1st r}). For constant $p,q\geq 0$, if $G$ is connected, 
$\ker\partial'=\ker\delta'$.
\end{lemma}
\begin{proof}
($\subseteq$) If $\omega\in\ker\partial'$, $(\partial'\omega)(\Xi;g)=0$ 
for all $(\Xi;g)\in\gg_p^q\times G$. Then, for $z\in\gg$, 
\begin{align*}
(\delta'\omega)(\Xi;z)=\Phi_{row}^q(\partial'\omega)(\vec{\gamma};z)=\frac{d}{d\tau}\rest{\tau =0}(\partial'\omega)(\Xi;\exp_G(\tau z))=0.
\end{align*}
($\supseteq$) Conversely, if $\omega\in\ker\delta'$, 
$(\delta'\omega)(\Xi;x)=\dot{\rho}_1(z)\omega(\Xi)=0$ for all 
$(\Xi;z)\in\gg_p^q\times\gg$. Being connected, $G$ is 
generated by $\exp_G(U)\subset G$ for some neighborhood of the identity $U$. 
Therefore, for all $g\in G$, there exist $z_1,...,z_n\in\gg$ such that
$g=\exp_G(z_1)...\exp_G(z_n)$. Since $\rho_1$ is a Lie group homomorphism, 
it follows from Eq.~(\ref{TheExpOfGL(phi)}) that  
\begin{align*}
\rho_1(\exp_G(z))\omega(\Xi)=\exp_{GL(\phi)_1}(\dot{\rho}_1(z))\omega(\Xi)=\sum_{n=0}^\infty\frac{(\dot{\rho}_1(z)\phi)^n}{(n+1)!}\dot{\rho}_1(z)\omega(\Xi)=0
\end{align*}
for all $z\in\gg$. Now, 
$\partial'\omega(\Xi;g)=\rho_1(g)\omega(\Xi)=\rho_1(\exp_G(z_1)...\exp_G(z_n))\omega(\Xi)$ 
and it follows from Eq.~(\ref{GpRho1Homo}) that
\begin{align*}
\rho_1(\exp_G(z_1)...\exp_G(z_n)) & =\rho_1(\exp(z_1)...\exp(z_{n-1}))+\rho_1(\exp_G(z_n))+\rho_1(\exp(z_1)...\exp(z_{n-1}))\phi\rho_1(\exp_G(z_n));						
\end{align*}
hence, the result follows from a simple induction.
\end{proof}
As $\Phi_{row}$ restricted to the first column of (\ref{VerLADbl}) is the 
identity and there are no cochains of negative degree, Lemma~\ref{tweak0} 
is a van Est type theorem in degree $0$ and a consequence of several pieces 
of Lie theory. In contrast, the following lemma is stated as a general result 
of homological algebra and implies naturally that if $G$ is $1$-connected, 
$\Phi_{rwo}$ induces isomorphism in degree $1$.
\begin{definition}\label{phiRel}
If $(C_1^\bullet,d_{C_1})$ and $(C_2^\bullet,d_{C_2})$ are equal complexes 
except in degree zero, and there is a map $\xymatrix{\phi:C_1^0 \ar[r] & C_2^0}$ 
such that $d_{C_1}=d_{C_2}\circ\phi$, then they are called $\phi$-related. 
\end{definition}
\begin{lemma}\label{tweak1}
Let $(C_1^\bullet,d_{C_1})$, $(C_2^\bullet,d_{C_2})$ be $\phi_C$-related 
complexes, and let $(D_1^\bullet,d_{D_1})$, $(D_2^\bullet,d_{D_2})$ be  
$\phi_D$-related complexes. Let 
$\xymatrix{\Phi_1:C_1^\bullet \ar[r] & D_1^\bullet}$ and 
$\xymatrix{\Phi_2:C_2^\bullet \ar[r] & D_2^\bullet}$ be maps of complexes 
that coincide except in degree zero, where 
\begin{align*}
\xymatrix{C_1^0 \ar[rd]^{d_{C_1}}\ar[dd]_{\phi_C}\ar[rr]^{\Phi_1} &    & D_1^0 \ar[rd]^{d_{D_1}}\ar'[d]_{\phi_D}[dd] &      \\
                                                    & C_1^1=C_2^1 \ar[rr]^{\quad\qquad\Phi_1=\Phi_2} &       & D_1^1=D_2^1. \\
          C_2^0 \ar[ur]_{d_{C_2}}\ar[rr]_{\Phi_2}                &    & D_2^0 \ar[ur]_{d_{D_2}}                        &     } 
\end{align*} 
If $\Phi_1$ induces an isomorphisms $H^1(C_1)\cong H^1(D_1)$, then 
$\Phi_2$ induces an isomorphism $H^1(C_2)\cong H^1(D_2)$.
\end{lemma}
\begin{proof}
Let $Z^1_{X_k}:=\xymatrix{\ker(d_{X_k}:X_k^1 \ar[r] & X_k^2)}$ and 
$B^1_{X_k}:=d_{X_k}(X_k^0)$, for $X\in\lbrace C,D\rbrace$ and 
$k\in\lbrace 1,2\rbrace$, and consider the maps of exact sequences
\begin{align*}
\xymatrix{ 0 \ar[r] & B^1_{C_k} \ar[d]_{\Phi_k\rest{B^1_{C_k}}}\ar[r] & Z^1_{C_k} \ar[d]_{\Phi_k}\ar[r] & H^1(C_k) \ar[d]_{[\Phi_k]}\ar[r] & 0 \\
		   0 \ar[r] & B^1_{D_k} \ar[r] & Z^1_{D_k} \ar[r] & H^1(D_k) \ar[r] & 0 }
\end{align*}
whose associated long exact sequence is
\begin{align*}
\xymatrix{ 0 \ar[r] & \ker(\Phi_k\rest{B^1_{C_k}}) \ar[r] & \ker\Phi_k \ar[r] & \ker[\Phi_k] \ar[r] & \coker(\Phi_k\rest{B^1_{C_k}}) \ar[r] & \coker\Phi_k \ar[r] & \coker[\Phi_k] \ar[r] & 0. }
\end{align*}
By hypothesis, $\ker[\Phi_1]$ and $\coker[\Phi_1]$ are trivial thus 
implying $\coker(\Phi_1\rest{B^1_{C_1}})\cong\coker\Phi_1$ and 
$\ker(\Phi_1\rest{B^1_{C_1}})=\ker\Phi_1$, which we interpret as 
$\ker\Phi_1\subset B^1_{C_1}$. Since for every element $x\in C_1^0$, 
$d_{C_1}(x)=d_{C_2}(\phi_C(x))$, we have got that 
$B^1_{C_1}\subseteq B^1_{C_2}$; therefore, 
$\ker(\Phi_2\rest{B^1_{C_2}})=\ker\Phi_1\cap B^1_{C_2}=\ker\Phi_1$. As a 
consequence, $\ker[\Phi_2]$ vanishes, so the induced map in cohomology is 
injective and we are left with the short exact sequence
\begin{align*}
\xymatrix{ 0 \ar[r] & \coker(\Phi_2\rest{B^1_{C_2}}) \ar[r] & \coker\Phi_2 \ar[r]  & \coker[\Phi_2] \ar[r] & 0.}
\end{align*}
Now, $d_{D_1}=d_{D_2}\circ\phi_D$ implies $B^1_{D_1}\subseteq B^1_{D_2}$, 
so there is a map of exact sequences
\begin{align}\label{finSeq}
\xymatrix{ 0 \ar[r] & \coker(\Phi_1\rest{B^1_{C_1}}) \ar[d]_{\alpha}\ar[r] & \coker\Phi_1 \ar[d]_{Id}\ar[r]  & 0 \ar[d]\ar[r]     & 0 \\
         0 \ar[r] & \coker(\Phi_2\rest{B^1_{C_2}}) \ar[r]   & \coker\Phi_2 \ar[r]    & \coker[\Phi_2] \ar[r] & 0, }
\end{align}
where, for $y\in D^0_1$, 
$\alpha(d_{D_1}(y)+\Phi(B^1_{C_1})):=d_{D_2}(\phi_D(y))+\Phi(B^1_{C_2})$. 
The long exact sequence of (\ref{finSeq}) tells us that $\alpha$ is an 
isomorphism and $\coker[\Phi_2]$ is trivial, so the induced map in 
cohomology is surjective.

\end{proof}
\begin{remark}
In the proof of Lemma~\ref{tweak1}, the inclusions $B^1_{X_1}\subseteq B^1_{X_2}$ 
also give rise to exact sequences
\begin{align*}
\xymatrix{ 0 \ar[r] & B^1_{X_1} \ar[d]\ar[r]  & Z^1_{C_1} \ar[d]_{Id}\ar[r] & H^1(X_1) \ar[d]\ar[r] & 0 \\
           0 \ar[r] & B^1_{X_2} \ar[r]        & Z^1_{C_2} \ar[r]            & H^1(X_2) \ar[r]       & 0  }
\end{align*}
out of whose long exact sequences one reads  
\begin{align}\label{dirSum}
H^1(X_1)\cong H^1(X_2)\oplus \frac{B^1_{X_2}}{B^1_{X_1}}.
\end{align}
What the proof of Lemma~\ref{tweak1} ultimately says is that the isomorphism 
$H^1(C)\cong H^1(D)$ is diagonal with respect to the direct sum decompositions 
of Eq.~(\ref{dirSum}).
\end{remark}
\begin{proposition}\label{TweakedVanEst}
For constant $q\geq 0$, if $G$ is $k$-connected, $H^n(\Phi_{row}^q)=(0)$ 
for all degrees $n\leq k$.
\end{proposition}
\begin{proof}
Since $G$ is the $s$-fibre of the group bundle 
$\xymatrix{\gg_p^q\times G \ar@<0.5ex>[r]\ar@<-0.5ex>[r] & \gg_p^q}$ 
and is $k$-connected by hypothesis, Theorem~\ref{Crainic-vanEstRephrased} 
implies the result for $1<n\leq k$. That there is an isomorphism in degree 
$0$ follows from Lemma~\ref{tweak0}. As for degree $1$, the result follows 
from Lemma~\ref{tweak1} after noticing that letting 
\begin{align*}
\phi_p^q & \xymatrix{:\bigwedge^q\gg_p^*\otimes W \ar[r] & \bigwedge^q\gg_p^*\otimes V} & (\phi_p^q\omega)(\Xi) & :=\phi(\omega(\Xi))\quad\text{for }\Xi\in\gg_p^q,
\end{align*} 
the $q$th rows of (\ref{protoVerDbl}) and (\ref{VerLADbl}) are 
$\phi_p^q$-related and the same holds for the Chevalley-Eilenberg complex 
of $\gg$ with values in (\ref{rRep}) and the $q$th row of 
the $p$-page $C^{p,\bullet}_\bullet(\gg_1,\phi)$.
\end{proof}
\begin{theorem}\label{res-vanEst} 
If $G$ is $k$-connected, $H^n_{tot}(\Phi_{row})=(0)$ for all degrees 
$n\leq k$.
\end{theorem}
\begin{proof}
We compute the cohomology of the (transposed) mapping cone double of 
$\Phi_{row}$ using the spectral sequence of its filtration by rows, whose 
first page is
\begin{align*}
\xymatrix{ & \vdots & \vdots & \vdots &  \\ 
           & H^0(\Phi_{row}^2) \ar[u] & H^1(\Phi_{row}^2) \ar[u] & H^2(\Phi_{row}^2) \ar[u] & \dots \\
E^{p,q}_1: & H^0(\Phi_{row}^1) \ar[u] & H^1(\Phi_{row}^1) \ar[u] & H^2(\Phi_{row}^1) \ar[u] & \dots \\
           & H^0(\Phi_{row}^0) \ar[u] & H^1(\Phi_{row}^0) \ar[u] & H^2(\Phi_{row}^0) \ar[u] & \dots }
\end{align*}
Since $\Phi_{row}$ is defined row-wise by van Est maps, the $q$th row of 
the transposed mapping cone double coincides with the mapping cone of 
$\Phi_{row}^q$. Invoking Proposition~\ref{TweakedVanEst}, the $q$th row 
of $E^{p,q}_1$ vanishes below $k$; indeed, $G$ is $k$-connected by hypothesis. 
Given that $E^{p,q}_1\Rightarrow H^{p+q}_{tot}(\Phi_{row})$, the result follows 
from Lemma \ref{BelowDiag}.
\end{proof}
\begin{remark}
One can prove results analog to those in Subsection~\ref{subsec-1stAprox} for 
the horizontal differential of (\ref{DblRepn}) yielding as a first approximation 
a map to the double complex associated to the other LA-groupoid in (\ref{LAGpds}). 
In that case, the ideas needed to prove that its cohomology vanishes parallel those 
of Lemmas~\ref{tweak0} and \ref{tweak1}. We opted for the present approach because, 
in the second approximation, one would need a van Est theory adapted to the 
subcomplex of multilinear cochains.
\end{remark}
We are ready to prove that the cohomology of the restriction $\Phi_p$ 
(\ref{Phi_p}) of the van Est map $\Phi$ (\ref{vanEstMap}) to the $p$-pages 
vanishes. 
\begin{theorem}\label{vanEst p-pages}
If $H$ and $G$ are both $k$-connected, $H^n_{tot}(\Phi_p)=(0)$ for all 
degrees $n\leq k$.
\end{theorem}
\begin{proof}
As in the proof of Theorem~\ref{2vE-vs}, it follows from the K\"unneth 
formula and the $k$-connectedness of $H$ and $G$, that $\G_p$ is 
$k$-connected as well. Theorem~\ref{Ver p-vanEst} and Proposition 
\ref{ConeCoh} imply that $\Phi_V$ induces isomorphism in cohomology for 
$n\leq k$, and it is injective for $n=k+1$. Analogously, as $G$ is 
$k$-connected, Theorem~\ref{res-vanEst} and Proposition~\ref{ConeCoh} imply 
that the same holds for $\Phi_{row}$. The result follows from 
Proposition~\ref{ConeCoh} after noticing that (cf. Eq.'s~(\ref{PhiVr}) and 
(\ref{PhiRow})), for constant $p$, 
\begin{align*}
\xymatrix{ C(\G_p^q\times G^r,W) \ar[rr]^{\Phi_p}\ar[dr]_{\Phi_V^r} & & \bigwedge^q\gg_p^*\otimes\bigwedge^r\gg^*\otimes W . \\
                       & C(G^r,\bigwedge^q\gg_p^*\otimes W) \ar[ur]_{\Phi_{row}^q} &  }
\end{align*}
\end{proof}
As announced in the introduction, we can now prove the main Theorem. 
\begin{proof}(\textit{of Theorem}~\ref{2-vanEstTheo})
We compute the cohomology of the mapping cone triple of $\Phi$ using the 
spectral sequence of the filtration by columns of (\ref{Cllpsed}), whose 
first page is (schematically)
\begin{align}\label{CllpsedCoh}
\xymatrix{ & \vdots & \vdots & \vdots & \vdots \\ 
           & H_{tot}^3(\Phi_0) \ar[r]\ar@{-->}[rrd]\ar@{.>}[rrrdd] & H_{tot}^3(\Phi_1) \ar[r]\ar@{-->}[rrd]\ar@{.>}[rrrdd] & H_{tot}^3(\Phi_2) \ar[r]\ar@{-->}[rrd] & H_{tot}^3(\Phi_3) \ar[r] & \dots \\
E^{p,q}_1: & H_{tot}^2(\Phi_0) \ar[r]\ar@{-->}[rrd] & H_{tot}^2(\Phi_1) \ar[r]\ar@{-->}[rrd] & H_{tot}^2(\Phi_2) \ar[r]\ar@{-->}[rrd] & H_{tot}^2(\Phi_3) \ar[r] & \dots \\
           & H_{tot}^1(\Phi_0) \ar[r] & H_{tot}^1(\Phi_1) \ar[r] & H_{tot}^1(\Phi_2) \ar[r] & H_{tot}^1(\Phi_3) \ar[r] & \dots\\
           & H_{tot}^0(\Phi_0) \ar[r] & H_{tot}^0(\Phi_1) \ar[r] & H_{tot}^0(\Phi_2) \ar[r] & H_{tot}^0(\Phi_3) \ar[r] & \dots }
\end{align}
By definition, the $p$th column in (\ref{CllpsedCoh}) is given by the 
cohomology of the mapping cone double of $\Phi_p$. Invoking 
Theorem~\ref{vanEst p-pages}, the $p$th column of $E^{p,q}_1$ vanishes 
below $k$; indeed, both $H$ and $G$ are $k$-connected by hypothesis. Given 
that $E^{p,q}_1\Rightarrow H^{p+q}_{tot}(\Phi)$, the result follows from Lemma 
\ref{BelowDiag} and Proposition~\ref{ConeCoh}.

\end{proof}
And having all the ingredients to run van Est's strategy:
\begin{corollary}
Every finite-dimensional Lie $2$-algebra is integrable.
\end{corollary} 

\acks This work was supported by CAPES; 
the \emph{National Council for Scientific and Technological Development} 
- CNPq; and FAPERJ [grant number E-26/202.439/2019].

\end{document}